\documentclass[journal]{IEEEtran}

\usepackage{cite}
\usepackage{amsmath,amssymb,amsfonts,amsthm}
\usepackage{algorithmic}
\usepackage{graphicx}
\usepackage{algorithm,algorithmic}
\usepackage{textcomp}

\usepackage{mathtools}  
\usepackage{dashbox} 
\usepackage{adjustbox} 
\usepackage{tikz-cd} 
\usepackage{comment}

\newtheorem{theorem}{Theorem}
\newtheorem{lemma}{Lemma}
\newtheorem{remark}{Remark}
\newtheorem{definition}{Definition}
\newtheorem{proposition}{Proposition}
\newtheorem{example}{Example}

\newtheorem{corollary}{Corollary}

\newcommand{\cred}[1]{{\color{red}{#1}}}

\def \S {\Sigma}

\def \bR {\mathbb{R}}

\begin{document}

\title{
EBIF: Exact Bilinearization Iterative Form for Control-Affine Nonlinear Systems}
\author{Yuan-Hung Kuan, \IEEEmembership{Member, IEEE}, and Jr-Shin Li, \IEEEmembership{Fellow, IEEE}
\thanks{The authors are with the Department of Electrical and Systems Engineering, Washington University in St. Louis, St. Louis, MO 63130 USA (e-mail: {\tt\small k.yuan-hung@wustl.edu, jsli@wustl.edu})}
}
\maketitle

\begin{abstract}
In this paper, we develop a novel framework, Exact Bilinearization Iterative Form (EBIF), for transforming a nonlinear control-affine system into an exact finite-dimensional bilinear representation. In contrast to most existing approaches which generally lead to an infinite-dimensional representation, the proposed EBIF approach yields an iterative procedure for constructing a finite set of smooth coordinate functions that define an embedding, enabling an exact bilinear representation of the original nonlinear dynamics.
Leveraging tools from algebra and differential geometry, we establish both necessary and sufficient conditions for a nonlinear system to be exactly bilinearizable. 
We further illustrate how the EBIF-induced bilinear systems facilitate reachability analysis and control design. Through theoretical analysis and numerical simulations, we demonstrate the effectiveness of the EBIF framework and highlight its potential in simplifying control synthesis for nonlinear systems.
\end{abstract}

\begin{IEEEkeywords}
Nonlinear control-affine systems, Exact bilinearization, Iterative algorithm, Smooth Embedding
\end{IEEEkeywords}

\section{Introduction} \label{sec:introduction}
\IEEEPARstart{T}{he} analysis and control of nonlinear systems remain a fundamental challenge in control theory due to the inherent complexity of their nonlinear behavior \cite{Isidori_Springer95, Slotine_ANC91, Krener_SCL83}. These classes of control systems are prevalent in engineering and natural sciences, governing a wide range of phenomena from fluid dynamics \cite{Slotine_ANC91} to biological processes \cite{Iqbal_NE17} and robotic systems \cite{Olfati_MIT01}. However, unlike linear systems, nonlinear systems lack many of the convenient structural properties that facilitate analysis and design, rendering fundamental concepts such as controllability and observability become highly nontrivial and often require case-specific tools and assumptions.
Among various approaches addressing this complexity, \textit{bilinearization}, the transformation of nonlinear systems into bilinear forms, has emerged as a powerful tool for enabling tractable analysis and control synthesis \cite{Schwarz_Springer88, Pardalos_Springer10,Lo_SIAM75,Deutscher_MCMDS05}. This transformation establishes a structured approximation or equivalence between the original nonlinear dynamic and a bilinear representation, enabling the use of analytical and computational tools that have been extensively developed for bilinear systems \cite{Tie_TAC23, Tie_JCO14, Breiten_SCL10, Bruni_TAC74, Brockett_Springer75, Brockett_Springer81, Brockett_springer73}.

Classical techniques such as Carleman linearization (or bilinearization) \cite{Tsiligiannis_JMAA89, Kowalski_WS91, Rauh_IFAC09, Amini_Arxiv22, Steeb_JMAA80, Rotondo_ECC22, Fang_CDC16} involves approximating a nonlinear system by an infinite-dimensional linear system through the application of polynomial expansions. This class of methods expresses the nonlinear analytic vector field as a Taylor series and constructs an augmented state space by stacking higher-order monomials of the original state variables. Although this approach provides a formal linear or bilinear structure, its practical use is limited by the rapid growth of system dimension and the truncation error introduced when approximating the infinite-dimensional system with a finite-dimensional one. 
Recent advances in Koopman operator theory have inspired new directions in bilinearization by leveraging observable functions to approximate nonlinear dynamics through linear or bilinear representations \cite{Mauroy_Springer20, Bevanda_ARC21, Goswami_ARC22, Bruder_RAL21, Narasingam_IJC23, Zhang_Arxiv22}. This class of approaches has emerged as a powerful framework for analyzing nonlinear systems by lifting the dynamics into a space of observable functions where the evolution is governed by an infinite-dimensional linear (or bilinear) operator. Notably, this framework enables global linear representations of nonlinear systems, which has been extensively studied in the context of system identification \cite{Mauroy_TAC19}, prediction \cite{Otto_SIAM24}, stabilization \cite{Narasingam_IJC23, Huang_Springer20}, and control \cite{Bevanda_ARC21, Goswami_ARC22, Bruder_RAL21}. 
Several recent works have proposed data-driven techniques to approximate the Koopman operator using finite-dimensional subspaces spanned by learned or pre-selected observables \cite{Brunton_Arxiv21, Otto_ARC21, Narasingam_IJC23, Goswami_ARC22, Kaiser_ML21}. While this offers significant flexibility, one major challenge remains: there is no guarantee that a finite-dimensional Koopman-invariant subspace exists that accurately captures the nonlinear system dynamics, especially when control inputs are present. 

In this paper, we propose the Exact Bilinearization Iterative Form (EBIF), a systematic framework that provides a necessary and sufficient condition for determining when a nonlinear control-affine system can be exactly represented as a finite-dimensional bilinear system. EBIF leverages the construction of invariant vector spaces using Lie derivatives and exploits the ascending chain property of vector space Noetherian to construct an exact bilinear embedding. In addition to its theoretical foundations, EBIF provides a systematic algorithm for constructing both the bilinear representation and the corresponding transformation map, enabling practical application in theoretic analysis and control design. We further illustrate how the resulting bilinear system can be effectively used and simplify reachability analysis and optimal control of the original nonlinear system.

The remainder of this paper is organized as follows. In Section II, we formally introduce the concept of exact bilinearization. Section III presents the theoretical formulation of EBIF framework including the Lie derivative and Lie invariance of modules and establishes the necessary and sufficient conditions for a nonlinear system to be exactly bilinearizable. Leveraging the EBIF framework, in Section IV, we analyze reachability for an exact bilinearizable nonlinear system. Finally, Section V illustrates the effectiveness of the proposed framework through various simulation studies, including systems stabilization and optimal control design for steering problems.


\section{Exact bilinearization of nonlinear control-affine systems} \label{sec:background}
In this section, we introduce the concept of \textit{exact bilinearization} for nonlinear control-affine systems. This notion, though aligned, is fundamentally different from existing techniques, such as the notable Carleman linearization or Koopman-based approaches. These techniques rely on lifting a finite-dimensional nonlinear system to an infinite-dimensional linear or bilinear system, which can then be approximated through finite-dimensional truncations \cite{Amini_Arxiv22, Rotondo_ECC22, Goswami_ARC22, Bruder_RAL21}. In contrast, the scenario of exact bilinearization transforms a nonlinear system into a bilinear form of higher, yet still finite, dimension. More importantly, the transformed bilinear system exactly reflects the dynamics of the original nonlinear system without any approximation. This feature, in turn, enables tractable and efficient systems-theoretic analysis and computation for nonlinear systems through the consideration of bilinear dynamics.

In this work, we focus on the control-affine nonlinear system defined on $\mathbb{R}^n$, described by
\begin{align} \label{eq:Sigma}
    \Sigma:\ \dot{x}(t) &= f(x(t)) + \sum_{i=1}^{m} u_i(t) g_i(x(t)) ,
\end{align}
where $x(t)\in \mathbb{R}^n$ denotes the state of the system and $u(t)=(u_1(t),\ldots,u_m(t))'\in\mathcal{U}\subset\mathbb{R}^{m}$ is the control input taking values from the admissible control set $\mathcal{U}$. Here, $f,g_i\in \mathfrak{X}(\mathbb{R}^n)$ are smooth drift and control vector fields, respectively, defined on $\mathbb{R}^n$, where $\mathfrak{X}(\mathbb{R}^n)$ denotes the space of all smooth vector fields on $\mathbb{R}^n$. Our development aims to construct a \emph{smooth embedding} $\Psi$ that transforms this nonlinear system $\Sigma$ into the following finite-dimensional bilinear system,
\begin{align} 
    \label{eq:Sigma_b}
    \Sigma_b:\dot{z}(t) &= Az(t) + \sum_{i=1}^{m}u_i(t) B_iz(t), \\
    &:= \bar{f}(z(t)) + \sum_{i=1}^{m}u_i(t) \bar{g}_i(z(t)), \nonumber 
\end{align}
where $z(t)=\Psi(x(t))\in \mathbb{R}^{r}$ is the transformed state of $x(t)$ evolving on the embedded submanifold $\Psi(\mathbb{R}^n)\subset \mathbb{R}^r$ with $n<r<\infty$. Here, $\bar f$ and $\bar g_i$ are the drift and control vector fields of $\Sigma_b$ defined on $\mathbb{R}^r$, and $A,B_i\in\mathbb{R}^{r \times r}$ are constant matrices. Note that as $\Psi$ maps the original state space $\mathbb{R}^n$ into a higher-dimensional space $\mathbb{R}^r$, it is referred to as a \textit{lifting map}, and $z(t)$ is the associated \textit{lifting state}. If such a ``lifting'' exists, then we say that the system $\Sigma$ in \eqref{eq:Sigma} is exactly bilinearizable, with $\Sigma_b$ being its bilinear representation, and refer to $(\Sigma,\Sigma_b)$ as an exact bilinerization pair.

\begin{definition}[Exact bilinearization] \label{def:exactbilinearizable}
    We say a control-affine nonlinear system $\Sigma$, as described in \eqref{eq:Sigma}, is \emph{exactly bilinearizable} if there exists a smooth embedding $\Psi:\mathbb{R}^{n}\hookrightarrow \mathbb{R}^{r}$ defined by $x \mapsto z = \Psi(x)$ for all $x\in \mathbb{R}^n$ with $n< r<\infty$, such that the transformed system admits a finite-dimensional bilinear representation $\Sigma_b$, as given in \eqref{eq:Sigma_b}.
\end{definition}

To illustrate this definition with a concrete example, we consider the control-affine system modeling unicycle dynamics, which is commonly encountered in the field of robotics \cite{Aicardi_RAM95, Kuan_CDC24}.

\begin{example} \label{ex:unicycle_derive}
    Consider the system describing the dynamic of a unicycle, given by
    \begin{align*}
        \dot{ x}(t) = u_1(t) \begin{pmatrix} \cos(x_3(t)) \\ \sin(x_3(t)) \\ 0\end{pmatrix} + u_2(t) \begin{pmatrix} 0 \\ 0 \\ 1 \end{pmatrix} ,
    \end{align*}
    where $x(t) = ( x_1(t), x_2(t), x_3(t) )^{'}\in \mathbb{R}^3$ is the state and $u(t) = ( u_1(t), u_2(t) )'\in\mathbb{R}^2$ is the control input. Defining the smooth embedding $\Psi: \mathbb{R}^3\hookrightarrow \mathbb{R}^6$ by $z(t) = \Psi(x(t)) = (x_1(t),x_2(t),x_3(t),\cos x_3(t),\sin x_3(t),1)^{'}$, we observe that the system $x(t)$ can be lifted to a bilinear representation $z(t)$ driven by the same control inputs of the form,   
    \begin{align*}
        \dot{z}(t) &= u_1(t) B_1 z(t) + u_2(t) B_2 z(t),
    \end{align*}
    where $B_1 = e_1e_4^{'}+e_2e_5^{'}$ and $B_2 = e_3e_6'-e_4e_5^{'}+e_5e_4^{'}$ are constant matrices, and $e_j\in\mathbb{R}^6$ denotes the standard vector with $1$ in the $j^{th}$ component and $0$ elsewhere. 
\end{example}

\begin{remark}
    Alternatively, one may define a nonlinear system $\S$ in \eqref{eq:Sigma} to be exactly bilinearizable if it can be transformed into an inhomogeneous bilinear system of the form,
    {\small
    \begin{align} \label{eq:bilinearsystem}
        \Sigma_b':\dot{z}(t) = (B_0z(t)+D_0) + \sum_{i=1}^{m} u_i(t) (B_iz(t)+D_i),
    \end{align}
    }where $B_i\in\mathbb{R}^{r \times r}$ and $D_i\in\mathbb{R}^{r \times 1}$, $i=0,1,\ldots,m$, are constant matrices. This definition is consistent with Definition \ref{def:exactbilinearizable} as, without loss of generality, we may disregard the terms $D_i$ in \eqref{eq:bilinearsystem} by introducing an auxiliary state $\bar{z}(t) = (z(t)^{'},1)^{'}$. This augmentation leads to the bilinear system in the auxiliary state $\bar{z}$ as
    \begin{align*} 
        \dot{ \bar{z}}(t) &= \begin{pmatrix} B_0 & D_0\\0 & 0 \end{pmatrix}\bar{z}(t) + \sum_{i=1}^{m} u_i(t)\begin{pmatrix} B_i & D_i\\0 & 0 \end{pmatrix}\bar{z}(t),
    \end{align*}
    which is in the form of \eqref{eq:Sigma_b}.
\end{remark}

To characterize the relationship between the vector fields in the original and transformed systems and facilitate our analysis, we introduce the notion of $\Psi$-related vector fields. 

\begin{definition}[$\Psi$-related vector fields] \label{def:exactbilinearizable}
    Given vector fields $\tau_1$ on $\mathbb{R}^n$ and $\tau_2$ on $\mathbb{R}^r$, they are said to be \emph{$\Psi$-related} if $\frac{\partial\Psi}{\partial x}\cdot\tau_1(x)=\tau_2(\Psi(x))$ for all $x\in \mathbb{R}^n$, where $\frac{\partial\Psi}{\partial x}$ is the Jacobian matrix of $\Psi$. In this case, the vector $\tau_2\doteq\Psi_*\tau_1$ is the (pointwise) \emph{pushforward} of $\tau_1$ by $\Psi$ \cite{Lee_Springer03}.
\end{definition}

\begin{proposition} \label{prop:psirelated}
    Suppose $(\Sigma,\Sigma_b)$ is an exact bilinerization pair under the smooth map $\Psi$. Then, $f$ and $\bar f$, as well as $g_i$ and $\bar g_i$ for each $i=1,\dots,m$, are $\Psi$-related. Equivalently, the following commutative diagram holds: 
    \begin{center}
        \begin{tikzcd}[column sep=large]
        \mathbb{R}^n \arrow[r,"\Psi"] \arrow[d,"X_i"'] & \mathbb{R}^r \arrow[d,"Y_i"]\\
        T\mathbb{R}^n \arrow[r,"\Psi_{\ast}"] & T\mathbb{R}^r
        \end{tikzcd}
    \end{center}
    
\end{proposition}
\begin{proof}
    Because $\Sigma_b$ is an exact bilinearization of $\Sigma$ by $\Psi$, we have $z(t)=\Psi(x(t))$ 
    for all $t$. Then, the chain rule yields 
        \begin{align}
         \dot z(t)&=\frac{d}{dt}\Psi(x(t))=\frac{\partial\Psi}{\partial x}\cdot\dot x(t)\nonumber\\
         &=\frac{\partial\Psi}{\partial x}\cdot\Big(f(x(t))+\sum_{i=1}^mu_i(t)g_i(x(t))\Big)\nonumber\\
         &=\Psi_*f(z(t))+\sum_{i=1}^mu_i(t)\Psi_*g_i(z(t)),\label{eq:bilinear_derivation}
    \end{align}
    where the last equality follows from the definition of pushforward vector fields. By comparing the system in \eqref{eq:bilinear_derivation} with $\Sigma_b$, we conclude that $\Psi_*f(z)=\bar f(z)=Az$ and $\Psi_*g_i(z)=\bar g_i(z)=B_iz$ as desired.    
\end{proof}

A fundamental question of theoretical and practical importance is then to investigate under what conditions a finite-dimensional bilinear system $\Sigma_b$ exists. To explore this question, in the next section, we will introduce the notion of invariance with respect to the module of vector fields generated by $f$ and $g_i$ of the nonlinear systems $\S$. This invariance property plays a crucial role in formulating the exact bilinearization conditions.

\section{Exact bilinearization induced by smooth coordinate functions} \label{sec:ebif}
 In this section, we introduce the Exact Bilinearization Iterative Form (EBIF) as the building block for constructing the necessary and sufficient exact bilinearization condition. The development of EBIF extensively exploits the algebraic structure of the space of smooth vector fields as modules over the ring of smooth functions. The main approach is to utilize Lie derivative iterations in EBIF as a systematic process to construct a sequence of expanding subspaces in which the vector fields governing the dynamics of a nonlinear system admit a linear representation.
\subsection{Coordinate representations of bilinear systems} 
\label{sec:coordinates}
The core of the EBIF for a nonlinear system is based on taking successive Lie derivative operations on an appropriate set of smooth functions along the vector fields governing the system dynamics. This iterative process systematically expands the chosen set of smooth functions, eventually constituting a coordinate system for the transformed system once the expansion stabilizes. Specifically, these newly generated coordinate functions place the transformed system in a bilinear representation, with both its drift and control vector fields being linear under these new coordinates. 


\subsubsection{Geometric construction of bilinear systems}
To further elaborate on the concept, let's consider a $\Psi$-related exact bilinerization pair $(\Sigma,\Sigma_b)$, as defined in Definition \ref{def:exactbilinearizable}. In this case, the nonlinear vector fields $f(x)$ and $g_i(x)$ governing $\Sigma$ are simultaneously transformed by $\Psi_*$ to higher-dimensional linear vector fields $Az$ and $B_iz$, respectively, governing $\Sigma_b$. In other words, each element (coordinate function) of $f$ and $g_i$ is linearized by $\Psi$. To express this geometrically, let $x=(x_1,\dots,x_n)$ and $z=(z_1,\dots,z_r)=(\Psi_1(x),\dots,\Psi_r(x))$ be the coordinate systems on $\mathbb{R}^n$ and $\mathbb{R}^r$, respectively. Then, the coordinate representations of the linear vector fields $\bar f(z)=Az$ and $\bar g_i(z)=B_iz$ are given by $\bar f(z)=\sum_{j=1}^r\sum_{l=1}^ra_{jl}z_l\frac{\partial}{\partial z_j}$ and $\bar g_i(z)=\sum_{j=1}^r\sum_{l=1}^rb_{jl}^iz_l\frac{\partial}{\partial z_j}$, where $a_{jl}$ and $b_{jl}^i$ are the $(j,l)$-entries of the matrices $A$ and $B_i$, respectively. Representing the $\Psi$-related property shown in Proposition \ref{prop:psirelated} using these coordinates further reveals the component-wise relationship between $f$ and $\bar f$, as well as $g_i$ and $\bar g_i$, as follows:

\begin{align*}
    \bar f=\sum_{j=1}^r\sum_{l=1}^n f_l\frac{\partial\Psi_j}{\partial x_l}\frac{\partial}{\partial z_j}=\sum_{j=1}^r\mathcal{L}_f\Psi_j\frac{\partial}{\partial z_j}
\end{align*}
and
\begin{align*}
    \bar g_i=\sum_{j=1}^r\sum_{l=1}^n g_{il}\frac{\partial\Psi_j}{\partial x_l}\frac{\partial}{\partial z_j}=\sum_{j=1}^r\mathcal{L}_{g_i}\Psi_j\frac{\partial}{\partial z_j},
\end{align*}
respectively, where $f=\sum_{l=1}^nf_l\frac{\partial}{\partial x_l}$ and $g_i=\sum_{l=1}^ng_{il}\frac{\partial}{\partial x_l}$ are the coordinate representations of $f$ and $g_i$, and $\mathcal{L}_f\Psi_j$ and $\mathcal{L}_{g_i}\Psi_j$ denote the Lie derivatives of $\Psi_j\in C^\infty(\mathbb{R}^n)$ with respect to $f$ and $g_i$, respectively. Here, the first equality follows from the property of $\Psi$-related vector fields, with the detailed derivations provided in Appendix Lemma \ref{lemma:coord}.
As a result, each component of the system $\Sigma_b$ in the $z$-coordinates, obeying
\begin{align}
    \label{eq:bilinear_coordinate}
    \dot{z}_j(t) = \mathcal{L}_{f}\Psi_j(x(t)) + \sum_{i=1}^m \mathcal{L}_{g_i}\Psi_j(x(t))
\end{align}
for $j=1,\ldots,r$, is linear in $z=\Psi(x)$. 
This derivation indicates that the coordinate transformation $\Psi$ ``stretches'' all the nonlinear $x$-coordinate functions, i.e., $f_l$ and $g_{il}$, simultaneously into linear $z$-coordinate functions. This motivates a general strategy toward developing exact bilinearization using coordinate functions of vector fields from nonlinear systems to generate linear vector fields that synthesize bilinear systems.

\subsubsection{Algebraic viewpoint of bilinearization}
\label{sec:algebra}
The above geometric construction provides a local derivation of the bilinearized system, involving the linearization of the nonlinear vector fields $f$ and $g_i$ at each point in $\mathbb{R}^n$. The global nature of bilinearization can be understood through an algebraic interpretation of vector fields. 
Because the tangent bundle $T\mathbb{R}^n$ of $\bR^n$ is globally trivialized as $\mathbb{R}^n\times\mathbb{R}^n$ \cite{Lee_Springer03}, the space of smooth vector fields on $\mathbb{R}^n$, denoted $\mathfrak{X}(\mathbb{R}^n)$, is a module over $C^\infty(\mathbb{R}^n)$ generated by this coordinate frame, where $C^\infty(\mathbb{R}^n)$ denotes the ring of real-valued smooth functions defined on $\mathbb{R}^n$. Hence, $\mathfrak{X}(\mathbb{R}^n)$ is isomorphic to $C^\infty(\mathbb{R}^n)\otimes\mathbb{R}^n$, where $\otimes$ denotes the tensor product of modules (vector spaces) over $\mathbb{R}$. In this algebraic setup, the embedding $\Psi:\mathbb{R}^n\rightarrow\mathbb{R}^r$ generates a vector bundle homomorphism $\frac{\partial\Psi}{\partial x}:T\mathbb{R}^n\rightarrow T\mathbb{R}^r$, which uniquely determines a module homomorphism from $C^\infty(\mathbb{R}^n)\otimes\mathbb{R}^n$ to $C^\infty(\mathbb{R}^r)\otimes\mathbb{R}^r$. Interpreting the bilinearization operation from this perspective, the action of $\frac{\partial\Psi}{\partial x}$ on the $C^\infty(\mathbb{R}^n)$-component maps those generated by the nonlinear coordinate functions $f_i(x)$ and $g_{ij}(x)$, for $i=1,\dots,n$ and $j=1,\dots,m$, to the ideal generated by linear functions $\sum_{j=1}^ra_{ij}z_j$ and $\sum_{j=1}^rb^k_{ij}z_j$, for $i=1,\dots,r$ and $k=1,\dots,m$ (supported on $\Psi(x)$). Parallel to the geometric construction, this algebraic viewpoint also suggests using the coordinate functions of the vector fields in the nonlinear system to generate the bilinearization transformation.

\subsection{Exact bilinearization iterative form} \label{subsec:EBIF}
Following the ideas developed in Section \ref{sec:algebra}, in this section, we will formally introduce EBIF, which provides a systematic machinery to iteratively construct a coordinate system that enables a bilinear representation $\S_b$ 
of system $\S$ 
using the nonlinear coordinate functions in the vector fields of system $\S$. Specifically, we will establish EBIF using the algebraic interpretation presented in Section \ref{sec:algebra}.

\begin{definition}[Lie derivation over modules] \label{def:LtauGamma}
    Let $\Gamma$ be the free module over a commutative ring $\mathbb{A}$, generated by a finite set $\{\gamma_1,\ldots,\gamma_d\}$ of functions in $C^\infty(\mathbb{R}^n)$, denoted by $\Gamma=\langle\gamma_1,\ldots,\gamma_d\rangle_{\mathbb{A}}$. The  \emph{Lie derivation} of $\Gamma$ with respect to a vector field $\tau\in\mathfrak{X}(\mathbb{R}^n)$ is defined as  
    \begin{align}
    \label{eq:Lie_derivation}
        \mathcal{L}_{\tau}\Gamma = \big\langle \mathcal{L}_{\tau}\gamma_1,\dots, \mathcal{L}_{\tau}\gamma_d \big\rangle_{\mathbb{A}},
    \end{align} 
    the free $\mathbb{A}$-module generated by $\mathcal{L}_\tau\gamma_i\in C^\infty(\mathbb{R}^n)$ for all $i=1,\dots,d$.
\end{definition}

Note that the Lie derivation defined above immediately gives rise to an $\mathbb{A}$-module homomorphism from  $\Gamma$ to $\mathcal{L}_\tau\Gamma$, denoted by $\mathcal{L}_\tau$ as well. This $\mathbb{A}$-module homomorphism is defined in the natural way as $\mathcal{L}_\tau\big(\sum_{i=1}^d \alpha_i\gamma_i\big)=\sum_{i=1}^d\alpha_i\mathcal{L}_\tau\gamma_i$ for any $\alpha_1,\ldots,\alpha_d\in\mathbb{A}$. Because  $\mathcal{L}_\tau\Gamma$ is generated by $\mathcal{L}_\tau\gamma_1$, $\dots$, $\mathcal{L}_\tau\gamma_d$, $\mathcal{L}_\tau$ is a subjective homomorphism (see more details in Lemma \ref{lemma:Kstationaryproof} in Appendix), equivalently, an \emph{epimorphism} in the category of $\mathbb{A}$-modules. On the other hand, both $\Gamma$ and $\mathcal{L}_\tau\Gamma$ are submodules of $\langle C^\infty(\mathbb{R}^n)\rangle_{\mathbb{A}}$, and hence it is possible to analyze their inclusion relation. 


\begin{definition}[Lie invariance]
    \label{def:invariant}
     Given $\tau\in\mathfrak{X}(\mathbb{R}^n)$, a finitely generated free $\mathbb{A}$-submodule $\Gamma$ of $\langle C^\infty(\mathbb{R}^n)\rangle_{\mathbb{A}}$ is said to be \emph{$\tau$-invariant} if $\mathcal{L}_\tau\Gamma\subseteq\Gamma$. Moreover, $\Gamma$ is \emph{$\mathcal{T}$-invariant} for a set $\mathcal{T}=\{\tau_1,\dots,\tau_m\}$ of vector fields in $\mathfrak{X}(\mathbb{R}^n)$ if $\Gamma$ is $\tau_i$-invariant for all $i=1,\dots,m$.     
\end{definition}

The Lie invariant property plays a crucial role in developing the exact bilinearization technique. In particular, the system in \eqref{eq:bilinear_coordinate} illuminates the Lie derivative representation of bilinearized vector fields. Therefore, for generating a finite-dimensional bilinearized system, it is necessary to construct a finitely generated Lie invariant submodule of $\langle C^\infty(\mathbb{R}^n)\rangle_{\mathbb{A}}$ that is generated by finitely many smooth functions. This reveals the main source of inspiration for our exact bilinearization technique, which gives a systematic approach to the construction of such a module.

The main idea of our exact bilinearization technique is to iteratively expand a non-Lie invariant submodule of $\langle C^\infty(\mathbb{R}^n)\rangle_{\mathbb{A}}$ using Lie derivation and internal direct sum operations, until the Lie invariance property is satisfied. For this reason, we refer to our technique as the \textit{exact bilinearization iterative form} (EBIF).

\begin{definition}[Exact bilinearization iterative form] \label{def:K_stationary}
    Let $\Gamma_0$ be a finitely generated free $\mathbb{A}$-submodule of $\langle C^\infty(\mathbb{R}^n)\rangle_{\mathbb{A}}$ and $\mathcal{T}$ be a finite set of smooth vector fields in $\mathfrak{X}(\mathbb{R})$, then we say $\{\Gamma_k\}_{k\in\mathbb{N}}$ is an EBIF sequence generated by $\mathcal{T}$ (initiated with $\Gamma_0$) if it is a sequence of free $\mathbb{A}$-submodules of $\langle C^\infty(\mathbb{R}^n) \rangle_{\mathbb{A}}$ generated by the \textbf{Exact Bilinearization Iterative Form (EBIF)}, defined by
        \begin{equation} \label{eq:ebif}
      \Gamma_k = \begin{cases}
         \ \Gamma_0, & k=0 \\
         \ \Gamma_{k-1} + \sum\nolimits_{\tau\in \mathcal{T}} \mathcal{L}_{\tau} \Gamma_{k-1}, & k\in\mathbb{N}
      \end{cases}~,
    \end{equation}
    where the sums are all internal direct sums of $\mathbb{A}$-modules, e.g., $\sum_{i=1}^m \mathcal{L}_{\tau_i}\Gamma_{k-1}=\big\{\sum_{i=1}^m\mathcal{L}_{\tau_i}\gamma_i:\gamma_i\in\Gamma_{k-1}\big\}$ for $\tau_1,\dots,\tau_m\in \mathcal{T}$. 
\end{definition}

    By the definition of internal direct sums,  $\Gamma_{k-1}+\mathcal{L}_\tau\Gamma_{k-1}$ necessarily contains both $\Gamma_{k-1}$ and $\mathcal{L}_\tau\Gamma_{k-1}$ as submodules. Therefore, EBIF generates an ascending chain of submodules of $\langle C^\infty(\mathbb{R}^n)\rangle_{\mathbb{A}}$. For any $i>j$, let $\iota_{ij}:\Gamma_{i}\hookrightarrow\Gamma_j$ be the embedding of $\Gamma_i$ into $\Gamma_j$, then the family of embeddings $\{\iota_{ij}\}_{0\leq i\leq j}$ satisfies $\iota_{ii}$ is the identity on $\Gamma_i$ and $\iota_{jk}\circ\iota_{ij}=\iota_{ik}$ for all $i\leq j\leq k$. This implies that the pairs $(\Gamma_i,\iota_{ij})$ form a \emph{direct system} of $\mathbb{A}$-modules, so that the direct limit is well-defined $\mathbb{A}$-module, given by,
    \begin{align}
    \label{eq:EBIF_module}
    \Gamma^{\ast} := \varinjlim \Gamma_k= \bigsqcup_{k=0}^{\infty} \Gamma_k\Big/\sim,
    \end{align}
    referred to $\Gamma^{\ast}$ as the \textbf{EBIF module}, where $\bigsqcup$ denotes the disjoint union and the equivalence relation is defined as $\gamma_i\sim\gamma_j$ if and only if $\iota_{ij}(\gamma_i)=\iota_{jk}(\gamma_j)$ for any $i\leq k$ and $j\leq k$. 
    
    The use of the Lie derivative operations in the generation of the EBIF sequence announces the candidacy for the use of the elements in the EBIF module $\Gamma^*$ as the coordinates for exact bilinearization. Because the exactly bilinearized system is required to be finite dimensional, a ``finiteness'' property is expected for $\Gamma^{\ast}$, which will be manifested in terms of the Noetherian property. Specifically, $\Gamma^{\ast}$, or generally any $\mathbb{A}$-module, is said to be a \emph{Noetherian $\mathbb{A}$-module}, if it satisfies the \emph{ascending chain condition}, that is, every ascending sequence of submodules of $\Gamma^{\ast}$ stabilizes. This particularly indicates that the EBIF sequence has to stabilize as $\Gamma_0\subseteq\Gamma_1\subseteq\cdots\subseteq\Gamma_{K^*}=\Gamma_{K^*+1}=\cdots=\Gamma^*$ for some integer $K^*$, and hence $\Gamma^*$ is necessarily a finitely generated $\mathbb{A}$-module. This further guarantees the Lie invariance of $\Gamma^*$ as follows.
    



    \begin{corollary}
    \label{thm:Lie_invariance}
        The EBIF module $\Gamma^*$ defined in \eqref{eq:EBIF_module} is $\mathcal{T}$-invariant, provided $\Gamma^*$ is Noetherian.
    \end{corollary}
    \begin{proof}
        The fact that $\Gamma^*$ is Noetherian implies the chain of submodules $\{\Gamma_k\}_{k\in\mathbb{N}}$ satisfies the ascending chain condition. This suggests that $\Gamma^{\ast} = \Gamma^{\ast} + \sum\nolimits_{\tau\in \mathcal{T}} \mathcal{L}_{\tau} \Gamma^{\ast}$, which implies $\mathcal{L}_{\tau}\Gamma^{\ast}\subseteq \Gamma^{\ast}$.
    \end{proof}

    As a glance at the necessity of the Noetherian property for exact bilinearization, we will show that EBIF modules generated by bilinear systems on $\mathbb{R}^n$ are indeed Noetherian. In such a case, as well as the subsequential discussion, we choose $\mathbb{A}=\mathbb{R}$ so that every component of the EBIF sequence, and hence the EBIF module, is a vector space over $\mathbb{R}$.

\begin{proposition} \label{Prop:bilinear_Noetherian}
   Consider the bilinear system $\Sigma_b$ in \eqref{eq:Sigma_b} whose dynamics is governed by the set of linear vector fields $\mathcal{T}=\{Az, B_1z,\ldots,B_mz\}$ in $\mathfrak{X}(\mathbb{R}^r)$. The EBIF module (vector space) $\Gamma^*$ generated by $\mathcal{T}$ initiated with $\Gamma_0={\rm span}\{ z_1,\dots,z_n\}$ is Noetherian and, in particular, the EBIF sequence $\{\Gamma_k\}_{k\in\mathbb{N}}$ satisfies $\Gamma_k=\Gamma_0$ for all $k\in\mathbb{N}$, where $z=(z_1,\dots,z_n)$ are the coordinate functions on $\mathbb{R}^n$.   
\end{proposition}
\begin{proof}
   By the definition of Lie derivatives, we have $\mathcal{L}_{Az}z_i=\nabla z_i\cdot Az=\sum_{j=1}^na_{ij}z_j\in \Gamma_0$ for any $i=1,\dots,n$, where $a_{ij}$ denotes the $(i,j)$-entry of $A$ for $i,j=1,\dots,n$. This yields $\Gamma_0$ is $Az$-invariant. Similar calculations show the $B_jz$-invariance of $\Gamma_0$ for any $j=1,\dots,n$. Therefore, we obtain $\mathcal{L}_{\mathcal{T}}\Gamma_0\subseteq\Gamma_0$ so that $\Gamma_1=\Gamma_0$. Inductively, we conclude $\Gamma_k=\Gamma_0$ for all $k$, which implies $\Gamma^*=\Gamma_0$ is an $n$-dimensional vector space over $\mathbb{R}$. Therefore, all of the vector spaces of $\Gamma^*$ are finite-dimensional, concluding that $\Gamma^*$ is Noetherian. 
\end{proof}


\subsection{Necessary and sufficient conditions for exact bilinearization}

Motivated by Proposition \ref{Prop:bilinear_Noetherian}, in this section, the Noetherian property will be further employed to establish necessary and sufficient exact bilinearization conditions.   



\begin{theorem}[Main theorem] \label{thm:suffcond_exactbilinear}
        Given the control-affine nonlinear system $\Sigma$ in \eqref{eq:Sigma} defined on $\mathbb{R}^n$, let $\{\Gamma_k\}_{k\in\mathbb{N}}$ be an EBIF sequence generated by the set of vector field $\mathcal{T}=\{f,g_1,\dots,g_m\}\subset\mathfrak{X}(\mathbb{R}^n)$ governing the system dynamics, then the system $\Sigma$ is exactly bilineariable if and only if the EBIF module $\Gamma^*=\varinjlim\Gamma_k$ satisfies 
        \begin{enumerate}
        \item[(1)] $\Gamma^*$ is Noetherian, equivalently, a finite-dimensional vector subspace of $C^\infty(\mathbb{R}^n)$. 
        \item[(2)] The map $\Psi:\mathbb{R}^n\rightarrow\mathbb{R}^r$ given by $x\mapsto (\gamma_1(x), \dots, \gamma_r(x))$ is a smooth embedding for a basis $\{\gamma_1,\dots,\gamma_r\}$ of $\Gamma^*$. 
        \end{enumerate}
\end{theorem}
\begin{proof}
            (\textit{Necessity}) Suppose that the system $\Sigma$ in \eqref{eq:Sigma} can be exactly bilinearized to the system $\Sigma_b$ in \eqref{eq:Sigma_b} by a smooth embedding $\Psi:\mathbb{R}^n\rightarrow\mathbb{R}^r$. Let $\Gamma_0={\rm span}\{\gamma_1,\dots,\gamma_r\}$ with $\gamma_i\in C^\infty(\mathbb{R}^n)$ the $i^{\rm th}$ component of $\Psi$ for each $i=1,\dots,r$, then Proposition \ref{Prop:bilinear_Noetherian} together with the $\Psi$-related property in Proposition \ref{prop:psirelated} shows that $\Gamma^*=\Gamma_0$. Therefore, both Conditions (1) and (2) are satisfied.

            (\textit{Sufficiency}) Suppose Conditions (1) and (2) hold, we will construct a smooth embedding which exactly bilinearizes the system $\Sigma$. 
            Because $\Gamma^*$ is a finite-dimensional vector space, there is a finite basis $\{\gamma_1,\dots, \gamma_r\}$. We define a function $\Psi:\mathbb{R}^n\rightarrow\mathbb{R}^r$ by $x=(x_1,\dots,x_n)\mapsto z=(z_1,\dots,z_r)=(\gamma_1(x),\dots,\gamma_r(x))$. The $z$-coordinate representation of the system $\Sigma$ is given by
            \begin{align*}
                \dot{z}_j(t) = \dot{\gamma}_j(t) = \mathcal{L}_{f} \gamma_j(x(t)) + \sum\nolimits_{i=1}^{m} u_i(t) \mathcal{L}_{g_i} \gamma_j(x(t)),
            \end{align*}
            for $j=1,\dots,r$. The $\mathcal{T}$-invariance property of $\Gamma^*$ proved in Corollary \ref{thm:Lie_invariance} indicates $\mathcal{L}_{f} \gamma_j,\mathcal{L}_{g_i} \gamma_j\in\Gamma^*$ so that $\mathcal{L}_{f} \gamma_j=\sum_{k=1}^ra_{jk}\gamma_k$ and $\mathcal{L}_{g_i} \gamma_j=\sum_{k=1}^rb^i_{jk}\gamma_k$ for some $a_{jk},b^i_{jk}\in\mathbb{R}$. This yields
            \begin{align*}
                \dot{z}_j(t) &= \sum\nolimits_{k=1}^ra_{jk}\gamma_k(x(t)) \\
                &\qquad\qquad\quad + \sum\nolimits_{i=1}^{m} u_i(t) \Big(\sum\nolimits_{k=1}^{r}b^i_{jk}\gamma_k(x(t))\Big),\\
                &=\sum\nolimits_{k=1}^ra_{jk}z_k(t)+\sum\nolimits_{i=1}^{m} u_i(t) \Big(\sum\nolimits_{k=1}^{r}b^i_{jk}z_k(t)\Big),
            \end{align*}
            which gives rise to a bilinear system defined on $\mathbb{R}^r$. Since $\psi$ is a smooth embedding by Condition (2), the above system is an exact bilinearization of the system $\Sigma$. 
\end{proof}

\begin{remark}
    This theorem can be readily extended to control-affine systems defined on $\mathbb{C}^n$. In the complex setting, the vector fields $\{f,g_1,\dots,g_n\}$ governing the dynamic of $\Sigma$ are analytic and the control inputs are complex-valued. Consequently, the EBIF sequence consists of finite-dimensional vector subspaces of the space of entire functions defined on $\mathbb{C}^n$. In this case, the necessary and sufficient exact bilinearization conditions shown in Theorem \ref{thm:suffcond_exactbilinear} condition remains valid with $\mathbb{R}^n$ replaced by $\mathbb{C}^n$. 
\end{remark}

As shown in the proof of Theorem \ref{thm:suffcond_exactbilinear}, the exact bilinearization map is constructed from the EBIF module $\Gamma^*$. As the direct limit of an EBIF sequence $\{\Gamma_k\}_{k\in\mathbb{N}}$, $\Gamma^*$ necessarily depends on the initial module $\Gamma_0$. Therefore, different choices of $\Gamma_0$ may result in different $\Gamma^*$, leading to different exactly bilinearized systems, as illustrated in the following example.

\begin{example} \label{ex:ebif}
    To verify the exact bilinearization conditions proposed in Theorem \ref{thm:suffcond_exactbilinear} along with the investigation into the impact of the initial module in the EBIF sequence on the bilinearized systems, we revisit the unicycle in Example \ref{ex:unicycle_derive}. The drift and control vector fields of the system are given by $g_1(x) = ( \cos x_3(t) , \sin x_3(t), 0)^{'}$ and $g_2(x) = (0, 0, 1)^{'}$, respectively. 
    We first choose $\Gamma_0={\rm span}\{x_1,x_2,x_3\}$, then taking successive Lie derivatives of functions in $\Gamma_0$ with respect to the vector fields $g_1$ and $g_2$ yields
    \begin{align*}
        \Gamma_1 &= \Gamma_0 + \mathcal{L}_{g_1}\Gamma_0+ \mathcal{L}_{g_2}\Gamma_0 \\
        &={\rm span}\big\{ x_1,x_2,x_3,\cos(x_3),\sin(x_3),1 \big\}\\
        \Gamma_2 &= \Gamma_1 + \mathcal{L}_{g_1}\Gamma_1+ \mathcal{L}_{g_2}\Gamma_1 \\
        &={\rm span}\big\{ x_1,x_2,x_3,\cos(x_3),\sin(x_3),1 \big\}=\Gamma_1.
    \end{align*}
    We observe that the EBIF sequence stabilizes at $\Gamma_1$ and hence $\Gamma^*=\Gamma_1$ is a 6-dimensional vector subspace of $C^\infty(\mathbb{R}^3)$, verifying Condition (1). On the other hand, the map $\Psi:\mathbb{R}^3\rightarrow\mathbb{R}^6$, constituting the basis of $\Gamma^*$ as $(x_1,x_2,x_3)\mapsto(x_1,x_2,x_3,\cos x_3,\sin x_3,1)$, is the graph of the smooth function $\psi:\mathbb{R}^3\rightarrow\mathbb{R}^3$, $(x_1,x_2,x_3)\mapsto(\cos x_3,\sin x_3,1)$. Therefore, $\Psi$ is necessarily a smooth embedding (see Corollary \ref{cor:EBIF} below for the proof) so that Condition (2) is also satisfied. This qualifies $\Psi$ as an exact bilinearization map for the unicycle, which is consistent with the calculation in Example \ref{ex:unicycle_derive}.

    Alternatively, if we pick $\widetilde \Gamma_0={\rm span}\{x_1,x_2,\cos x_3,\sin x_3\}$, then $\widetilde\Gamma_0=\widetilde\Gamma_0+\mathcal{L}_{g_1}\widetilde\Gamma_0+ \mathcal{L}_{g_2}\widetilde\Gamma_0$ holds so that $\widetilde\Gamma^*=\widetilde\Gamma_0$ is a 4-dimensional vector space. To show the map $\widetilde\Psi:\mathbb{R}^3\rightarrow\mathbb{R}^4$, $(x_1,x_2,x_3)\mapsto(x_1,x_2,\cos x_3,\sin x_3)$ is a smooth embedding, we note that the last two components $(\cos x_3,\sin x_3)$ give rise to an embedding of the unit circle $\mathbb{S}^1$ into $\mathbb{R}^2$ with $x_3$ served as the angle coordinate. Therefore, $\widetilde\Psi$ defines a global coordinate chart of $\mathbb{R}^2\times\mathbb{S}^1$ as a smooth manifold embedded in $\mathbb{R}^4$, and hence generically a smooth embedding. The corresponding bilinearized system is 
    \begin{align*}
        \dot z=u_1(t)\left[\begin{array}{cccc} 0 & 0 & 1 & 0 \\ 0 & 0 & 0 & 1 \\ 0 & 0 & 0 & 0 \\ 0 & 0 & 0 & 0 \end{array}\right]z+u_2(t)\left[\begin{array}{cccc} 0 & 0 & 0 & 0 \\ 0 & 0 & 0 & 0 \\ 0 & 0 & 0 & -1 \\ 0 & 0 & 1 & 0 \end{array}\right]z,
    \end{align*}
    which is a subsystem of the bilinear system derived from $\Psi$ in Example \ref{ex:unicycle_derive}. This is because the basis of $\widetilde\Gamma^*$ is a subset of that of $\Gamma^*$. 

    Although $\Gamma_0$, constructed by using the coordinate functions of $\mathbb{R}^n$, does not generate a ``minimal bilinearization'', it provides a general guidance on the choice of the initial module in the EBIF sequence. This in turn gives rises to a sufficient exactly bilinearization condition.

\end{example}

\begin{corollary}\label{cor:EBIF}
    Consider the nonlinear control-affine system $\Sigma$. Let $\Gamma_0={\rm span}\{x_1,\dots,x_n\}$, the subspace of $C^\infty(\mathbb{R}^n)$ spanned by the coordinate functions of $\mathbb{R}^n$, be the initial module in the EBIF sequence $\{\Gamma_k\}_{k\in\mathbb{N}}$, then $\Sigma$ is exactly bilinearizable if the EBIF module $\Gamma^*=\varinjlim \Gamma_k$ is Noetherian. 
\end{corollary}
\begin{proof}
    It suffices to examine Condition (2) in Theorem \ref{thm:suffcond_exactbilinear}. Because $\Gamma_0\subseteq\Gamma^*$, we have $x_1,\dots,x_n\in\Gamma^*$ as well. The Noetherian assumption of $\Gamma^*$ implies the existence of $\gamma_{n+1},\dots,\gamma_r\in C^\infty(\mathbb{R}^n)$ such that $\Gamma^*={\rm span}\{x_1,\dots,x_n,\gamma_{n+1},\dots,\gamma_r\}$. It then remians to show that $\Psi:\mathbb{R}^n\rightarrow\mathbb{R}^r$, given by $(x_1,\dots,x_n)\mapsto(x_1,\dots,x_n,\gamma_{n+1},\dots,\gamma_r)$ is a smooth embedding, meaning, a topological embedding and an injective immersion. To this end, by defining the smooth map  $\psi:\mathbb{R}^n\rightarrow\mathbb{R}^{r-n}$ as $(x_1,\dots,x_n)\mapsto(\gamma_{n+1},\dots,\gamma_r)$, we have that $\Psi(x)=(x,\psi(x))\in\mathbb{R}^{n}\times\mathbb{R}^{r-n}$ is essentially the graph of $\psi$. Let $\pi:\mathbb{R}^{n}\times\mathbb{R}^{r-n}\rightarrow\mathbb{R}^n$ be the projection onto the $\mathbb{R}^n$-component, then $\pi\circ\Psi(x)=x$ holds for all $x\in\mathbb{R}^n$ so that $\pi$ is an inverse of $\Psi$, which implies that $\Psi$ is a topological embedding. Moreover, the Jacobian of $\Psi$ is given by $\nabla\Psi(x)=[\, I \mid \nabla\psi(x)\,]$ with $I$ the $n$-by-$n$ identity matrix, which is full rank at all $x\in\mathbb{R}^n$. Therefore, $\Psi$ is also an immersion, concluding the proof. 
\end{proof}

\section{EBIF-enabled Reachability Analysis
} 
\label{sec:reachability}

Identifying reachable sets of control systems is of fundamental importance with broad application domains ranging from cybersecurity and robotics to safety verification. However, systematic approaches to analytic computations of reachable sets for nonlinear systems remain underexplored. In this section, we will adopt the proposed EBIF to study reachable sets of exactly bilinearizable systems through their bilineared forms.



To start with, we first emphasize the necessity for the reachability analysis of exactly bilinearized systems by showing that they are never controllable.  

\begin{example} \label{ex:liftingexample}
    Consider the following control-affine nonlinear system described by
    \begin{align*}
        \frac{d}{dt} \begin{pmatrix} x_1(t) \\ x_2(t) \end{pmatrix} = \begin{pmatrix} x_1(t) \\ x_2(t)-x_1^2(t) \end{pmatrix} +\begin{pmatrix} u_1(t)\\u_2(t)\end{pmatrix},
    \end{align*}
    where $x(t) = ( x_1(t), x_2(t) )^{'}\in \mathbb{R}^2$. Clearly, this system is controllable on $\mathbb{R}^2$.
    Applying EBIF with $\Gamma_0 = \text{span}\{ x_1,x_2,x_1^2\}$ leads to an embedding $\Psi: \mathbb{R}^2\hookrightarrow \mathbb{R}^3$ defined by $z(t) = \Psi(x) = (x_1,x_2,x_1^2)^{'}$, with the corresponding bilinear representation given by
    \begin{align*}
        \dot{z}(t) = Az(t) + u_1(t)B_1z(t) +Du(t),
    \end{align*}
    where
    \begin{align*}
        A = \begin{pmatrix} 1 & 0 & 0 \\0&1&-1\\0&0&2  \end{pmatrix},~B_1=\begin{pmatrix} 0 & 0 & 0 \\0&0&0\\2&0&0  \end{pmatrix},~D=\begin{pmatrix} 1 & 0  \\0&1\\0&0  \end{pmatrix}.
    \end{align*}
    However, this bilinearized system is not controllable on $\mathbb{R}^3$ as the point $(0,0,1)'$ is not on the embedded submanifold $\Psi(\mathbb{R}^2)$ where the bilinear system evolves on (see Fig. \ref{fig:reachability}(a)).   
\end{example}
\begin{figure}[!ht]
    \centering
    \includegraphics[width=1\linewidth]{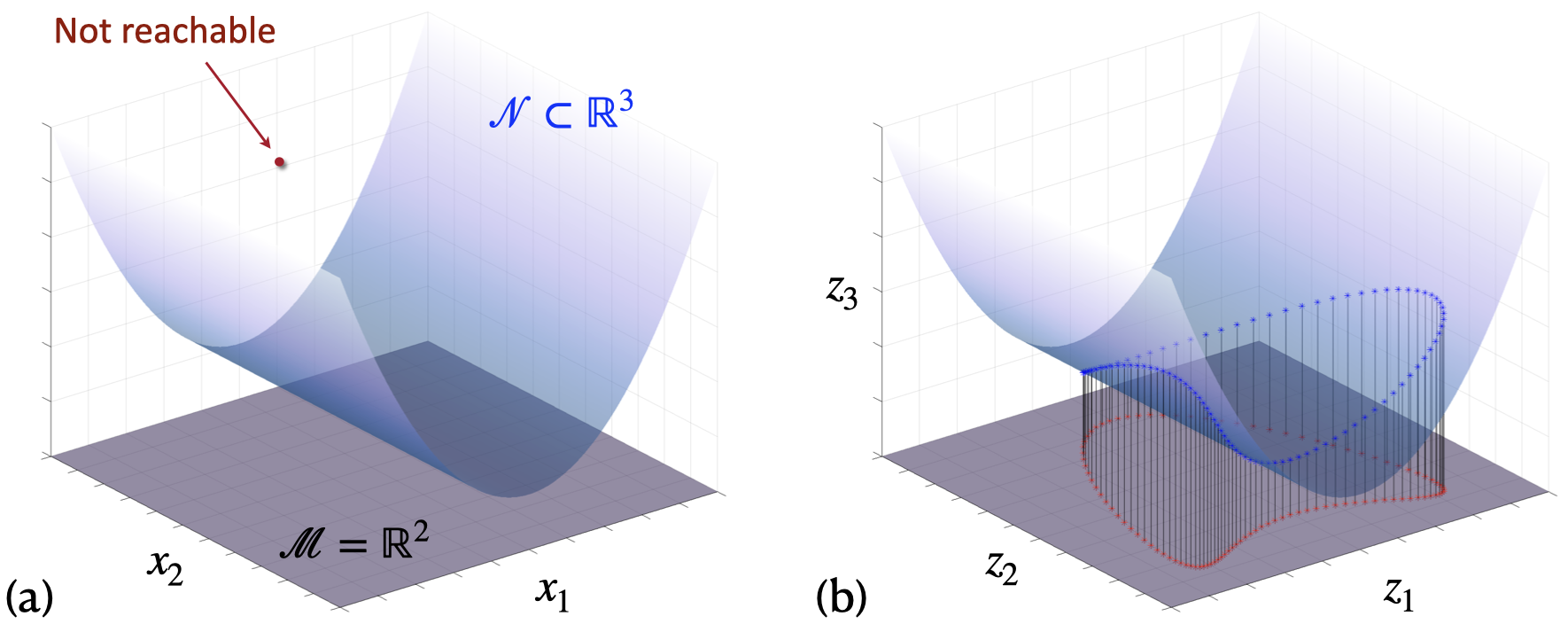}
    \caption{Illustration of (a) the original controllable manifold ${ M } = \mathbb{R}^2$ and the embedded submanifold $N = \Psi(M)\subset\mathbb{R}^3$ for the nonlinear system and its associated bilinear system; and (b) the reachable set at $T=1$ of the nonlinear system (marked in red) obtained by projecting the reachable set of its bilinear system (marked in blue) for the systems in Example \ref{ex:liftingexample} under control input $u\in\mathcal{U} = \{u(t) = C\in\mathbb{R}^2;~\|C\|=2 ;~ t\in[0,1]\}$.}
    \label{fig:reachability}
\end{figure}

Example \ref{ex:liftingexample} draws a general conclusion on uncontrolliabillity of an exactly bilinearized system in the case where the EBIF-embedding is not surjective. For this reason, we focus on reachability, instead of controllability, analysis of the bilinear system and use its reachable set to identify that of the original nonlinear system. Our main idea is to leverage the EBIF-embedding to represent the reachable set of the nonlinear system in terms of the \emph{Myhill semigroup}, that is, the space of flows generated by piecewise constant control inputs with the function composition operation, of the bilinear system. It is renowned that the reachable set of any control system can be identified with its Myhill semigroup \cite{Brockett_SIAM72, Brockett_springer73}. However, the semigroup elements for nonlinear systems generally fail to have closed form representations. A key advantage of bilinear systems is that any element in their Myhill semigroups admits a matrix representation. This then recognizes a major significance of the proposed EBIF framework, enabling a more tractable reachability analysis for nonlinear control systems.

\begin{lemma} \label{thm:reachableset}
   Let $(\Sigma,\Sigma_b)$ be an exact bilinearization pair under the smooth embedding $\Psi:\mathbb{R}^n\rightarrow\mathbb{R}^r$. Then, provided $[\text{ad}_A^k B_i,B_j]=0$ for all $i, j = 1, \ldots, m$ and $k= 0,\ldots,r^2-1$, the reachable set of the nonlinear system $\Sigma$ from any initial condition $x_0\in\mathbb{R}^n$ at any final time $T>0$ is given by
   \begin{align*}
        \mathcal{R}_T(x_0) &= \Big\{\xi\in \mathbb{R}^n \mid \xi = \mathcal{P}_n\big(e^{AT}e^H\Psi(x_0)\big)\text{ for }H\in\mathfrak{h} \Big\},
    \end{align*}
    where $\mathfrak{h}={\rm span}\{{\rm ad}^k_AB_i:i=1,\dots,m\text{ and }k=0,\dots,r^2-1\}$ and  $\mathcal{P}_n:\Psi(\mathbb{R}^n)\to\mathbb{R}^n$ is the left inverse of $\Psi$, i.e., $\mathcal{P}_n\circ\Psi(x) = x$ for all $x\in\mathbb{R}^n$. 
    Furthermore, given any control input $u$, the time-$t$ flows $\Phi_t^x:\mathbb{R}^n\rightarrow\mathbb{R}^n$ and $\Phi^z_t:\mathbb{R}^r\rightarrow\mathbb{R}^r$ of the systems $\Sigma$ and $\Sigma_b$, respectively, satisfy $\Phi_x^z=\mathcal{P}_n\circ\Phi_t^z\circ\Psi$ for all $t\in[0,T]$. 

\end{lemma}
\begin{proof}
    Consider the system defined on ${\rm GL}(r,\mathbb{R})$, the Lie group of invertible $r$-by-$r$ matrices, given by, $\dot\Phi(t)=\big(A+\sum_{i=1}^mu_i(t)B_i\big)\Phi(t)$ with the initial condition $\Phi(0)=I$, then the solution of the bilinear system $\Sigma_b$ satisfies $z(t)=\Phi(t)z(0)$.

    According to Lemma \ref{lemma:brockettreachability} in Appendix, the reachable set of the bilinear system $\Sigma_b$ on $\mathbb{R}^r$

    Consider the matrix differential equation associated with $\Sigma_b$ in $\mathbb{R}^{r\times r}$ in the form of
    $\dot{\Phi}(t) = A\Phi(t) + \sum\nolimits_{i=1}^{m} u_i(t) B_i\Phi(t)$ 
    with $\Phi(0) = I$. According to Lemma \ref{lemma:brockettreachability} in Appendix, a state $z(T)$ is reachable from $z(0)$ at $T>0$ for $\Sigma_b$ if it can be expressed as $z(T) = \Phi(T)z_0 = e^{AT}e^H\Phi(0)z_0$ for some $H \in \text{span}\{ \text{ad}_A^k B_i\}$. Since $z=\Psi(x)$, it follows that $\Psi(x(T)) = e^{AT}e^H\Psi(x_0)$ and thus the reachable states of the original system at final time $T$ take the form
    \begin{align*}
        x(T) = \mathcal{P}_n( \Psi(x(T)) ) = \mathcal{P}_n( e^{AT}e^H\Psi(x_0) ).
    \end{align*}
    Finally, the solution to the bilinear system $\Sigma_b$ admits the flow as $z(t) = e^{At}e^H z_0 := \Phi_t^z(z_0)$. This result together with the above representation of reachable state at time $t$ yields the desired result.
\end{proof}

The result of Lemma \ref{thm:reachableset} suggests that the reachable set of the original nonlinear system can be identified by first finding the reachable set of its associated (exact) bilinear system and then considering the associated ``projection'' to obtain the reachable set in the original system.
To illustrate and visualize this process, an example of finding the reachable set of the same nonlinear system given in Example \ref{ex:liftingexample} is provided in Fig. \ref{fig:reachability}(b) under the admissible control set $\mathcal{U} = \{u: u(t) =C\in\mathbb{R}^2 \text{ constant over }[0,1];~ \|C\|= 2; ~t\in[0,1]\}$.
\begin{remark}
    It is worth mentioning that a related result in \cite{Brockett_IEEE76} suggests that, for a general nonlinear control-affine system $\Sigma$ satisfying the \emph{involutive condition}, the reachable set can be expressed as $\mathcal{R}_T(x_0) = \left\{\xi\in\mathbb{R}^n\big| \xi = \{\exp(L_0)\}_G \exp(Tf) x_0 \right\}$.
    In other words, $\Phi_T^x(x_0)= \{\exp(L_0)\}_G \exp(Tf)x_0$. While this expression aligns conceptually with our formulation, the EBIF framework offers a simplification by requiring only matrix Lie brackets for its computation. This highlights the practical advantage of EBIF in making reachability analysis more tractable and computationally accessible.
\end{remark}

This result highlights that one of the major advantages of the proposed EBIF is that the reachable set of a control-affine nonlinear system can be explicitly expressed through the corresponding bilinear system using matrix semigroup.
Moreover, this exact bilinear transformation allows us to build a connection between the controllable manifold of the nonlinear system and the embedded submanifold of its associated bilinear system, establishing the dimension equivalences between these two submanifolds. 
To be more specific, let $\{\tau_1,\ldots,\tau_m\}\subset \mathfrak{X}(\mathbb{R}^n)$ be a family of vector fields on $\mathbb{R}^n$, then we denote the Lie algebra generated by the set of vector fields $\{\tau_1,\ldots,\tau_m\}$ evaluated at $x\in\mathbb{R}^n$ by $\text{Lie}_x\{\tau_1,\ldots,\tau_m\}$, that is, the smallest linear subspace of $T_x\mathbb{R}^n$ which contains $\{\tau_1,\ldots,\tau_m\}$ and is closed under the Lie bracket operation. Utilizing the $\Psi$-related property introduced in Proposition \ref{prop:psirelated}, the dimensional equivalence relation can be characterized through the Lie algebra generated by the set of vector fields as follows. 

\begin{lemma}\label{lemma:dimpreserve2}
    Suppose a control-affine nonlinear system $\Sigma$ is exactly bilinearizable with the corresponding bilinear representation $\Sigma_b$ obtained from a smooth embedding $\Psi:\mathbb{R}^n \to \mathbb{R}^r$. 
    Then it follows that, for all $x\in\mathbb{R}^n$, 
    $$\dim (\text{Lie}_x\{f,g_1,\ldots,g_m\}) = \dim (\text{Lie}_{\Psi(x)}\{\bar f,\bar g_1,\ldots,\bar g_m\}),$$
    where $\{f,g_1,\ldots,g_m\}\subset \mathfrak{X}(\mathbb{R}^n)$ and $\{\bar f,\bar g_1,\ldots,\bar g_m\}\subset \mathfrak{X}(\mathbb{R}^r)$, respectively. 
\end{lemma} 
\begin{proof}
    To simplify the notations, let $f:=g_0$ and $\bar f:=\bar g_0$, then it holds that $[\bar{g}_i , \bar{g}_j]_{\Psi(x)} = \Psi_{\ast} [g_i , g_j]_x$ for all $0\leq i\neq j \leq m$ following Lemma \ref{lemma:dimpreserve} in Appendix, which leads to the desired result.
\end{proof}
This result suggests that the dimension of the Lie algebra evaluated at $x\in\mathbb{R}^n$ is equivalent to the dimension of the Lie algebra evaluated at $\Psi(x)\in\mathbb{R}^r$ on the embedded submanifold $\Psi(\mathbb{R}^n)$, and consequently, the dimension of the controllable submanifold is preserved under the smooth embedding $\Psi$.
Furthermore, Lemma \ref{lemma:dimpreserve2} implies that the dimension of the controllable submanifold of the original nonlinear system can be simply obtained via the embedded submanifold of the associated bilinear system with computation involving only matrix Lie brackets operations. To illustrate this concept, we revisit the example presented in Example \ref{ex:unicycle_derive} and \ref{ex:ebif}. In this case, we compute the Lie bracket
\begin{align*}
    [\bar{g}_1, \bar{g}_2]_z = (\sin(x_3),-\cos(x_3),0,0)^{'} = \Psi_{\ast}[g_1,g_2]_x
\end{align*}
where $\bar{g}_i|_z = B_iz$ for $i=1,2$. Then, we observe that the Lie algebra generated by the vector fields $\{\bar{g}_i\}$ has dimension three, i.e., $\dim(\text{Lie}\{\bar{g}_1, \bar{g}_2\}) = 3$. This matches the dimension of the Lie algebra generated by the original nonlinear vector fields $\{g_i\}$, for which $\dim(\text{Lie}\{g_1, g_2\}) = 3$. 

\section{Simulations} \label{sec:simulation}
Leveraging the proposed EBIF framework, the theoretical developments in the preceding sections not only establish a foundation for deriving the bilinear system associated with a given control-affine nonlinear system but also highlight the advantages of utilizing the bilinear representation to analyze the nonlinear system. In this section, we illustrate the applicability and effectiveness of the EBIF framework with various applications, including reachability analysis, stabilization and optimal control design.

\subsection{Reachability analysis of equivalent systems}
In the first scenario, we demonstrate that the reachable sets of the two equivalent systems are identical by showing the evolution of the nonlinear system exactly matches the projection of the trajectory of its associated bilinear system.
To put the idea into a concrete setting, we begin with considering the following control-affine nonlinear system, which frequently arises in Koopman operator-based approaches \cite{Goswami_ARC22, Brunton_Arxiv21, Bevanda_ARC21}, given by
    \begin{align} \label{eq:simulation1}
        \frac{d}{dt} \begin{pmatrix} x_1(t) \\ x_2(t) \end{pmatrix} = \begin{pmatrix}
         \lambda_1 x_1 \\ \lambda_2 x_2+(2\lambda_1-\lambda_2)\lambda_3x_1^2 
         \end{pmatrix}, 
    \end{align}
    where $x(t) = ( x_1(t), x_2(t) )^{'}\in { M } = \mathbb{R}^2$ and $\lambda_i$ are fixed parameters. Note that the system given in Example \ref{ex:liftingexample} is a special case of this generalized system with $\lambda =\{1,1,1\}$. 
    Applying EBIF with $\Gamma_0 = \text{span}\{x_1,x_2,x_1^2\}$ yields an embedding $\Psi: { M }\hookrightarrow { N } \subset\mathbb{R}^3$ defined by $z = \Psi(x) = (x_1,x_2,x_1^2)^{'}$, with the corresponding bilinear representation given by
    \begin{align*}
        \dot{z}(t) = \begin{pmatrix} \lambda_1 & 0 & 0 \\ 0 & \lambda_2 & (2\lambda_1-\lambda_2)\lambda_3\\0 & 0 & 2\lambda_1  \end{pmatrix}z(t) .
    \end{align*}
    This representation suggests that each point on the solution trajectory of \eqref{eq:simulation1} can be identified from the solution to the above linear equation as 
    \begin{align*}
        x(t) &= \begin{pmatrix} 1 & 0 & 0 \\ 0 & 1 & 0 \end{pmatrix} \begin{pmatrix} e^{\lambda_1 t} & 0 & 0 \\ 0 & e^{\lambda_2 t} & \lambda_3(e^{2\lambda_1 t}-e^{\lambda_2 t})\\0 & 0 & e^{2\lambda_1 t}  \end{pmatrix}z(0) \\
        &= \begin{pmatrix} e^{\lambda_1 t} & 0 & 0 \\ 0 & e^{\lambda_2 t} & \lambda_3(e^{2\lambda_1 t}-e^{\lambda_2 t}) \end{pmatrix}\begin{pmatrix} x_1(0) \\ x_2(0) \\ x_1^2(0) \end{pmatrix},
    \end{align*}
    which leads to the explicit solutions
    \begin{align*}
        x_1(t) &= e^{\lambda_1 t} x_1(0),\\
        x_2(t) &= e^{\lambda_2 t} x_2(0) + \lambda_3(e^{2\lambda_1 t}-e^{\lambda_2 t}) x_1^2(0).
    \end{align*}
    These expressions match the exact solution obtained by directly solving \eqref{eq:simulation1}, demonstrating the validity of the EBIF framework. 
    Figure \ref{fig:simulationcompare}(a) illustrates the numerical simulation of the trajectory of the nonlinear system \eqref{eq:simulation1} (highlighted in blue) with the initial condition $x(0) = (0.1, 0.1)^{'}$, alongside the trajectory of its corresponding bilinear system (highlighted in red). 


Additionally, as discussed in Section \ref{subsec:EBIF}, EBIF with different initial choices result in different embeddings and thus different bilinear representations for a given nonlinear systems. In this same example \eqref{eq:simulation1}, applying EBIF with $\tilde\Gamma_0 = \text{span}\{x_1,x_2-\lambda_3x_1^2,x_1^2\}$ yields an alternative embedding $\tilde\Psi: { M }\hookrightarrow { N } \subset\mathbb{R}^3$ defined by $\tilde{z}(t) = \tilde\Psi(x) = (x_1,x_2-\lambda_3x_1^2,x_1^2)^{'}$, with the corresponding bilinear representation $\dot{\tilde{z}}(t) = \tilde{A}\tilde{z}(t)$ and submersion $\tilde{\mathcal{P}}_n$ given by
\begin{align*}
    \tilde{A} = \begin{pmatrix} \lambda_1 & 0 & 0 \\ 0 & \lambda_2 & 0\\0 & 0 & 2\lambda_1  \end{pmatrix}, \quad \tilde{\mathcal{P}}_n(\bar z) = \begin{pmatrix} \tilde z_1 \\ \tilde z_2 + \lambda_3\tilde z_3  \end{pmatrix},
\end{align*}
leading to the exact identical explicit solution given above. 
These results highlight that the bilinear representation of an exact bilinearizable nonlinear system is not unique and emphasize that different embeddings leads to bilinear systems evolving on different embedded submanifolds (see Fig. \ref{fig:simulationcompare}), even though they correspond to the same underlying nonlinear system.

\begin{figure}[!h]
    \centering
    \includegraphics[width=\linewidth]{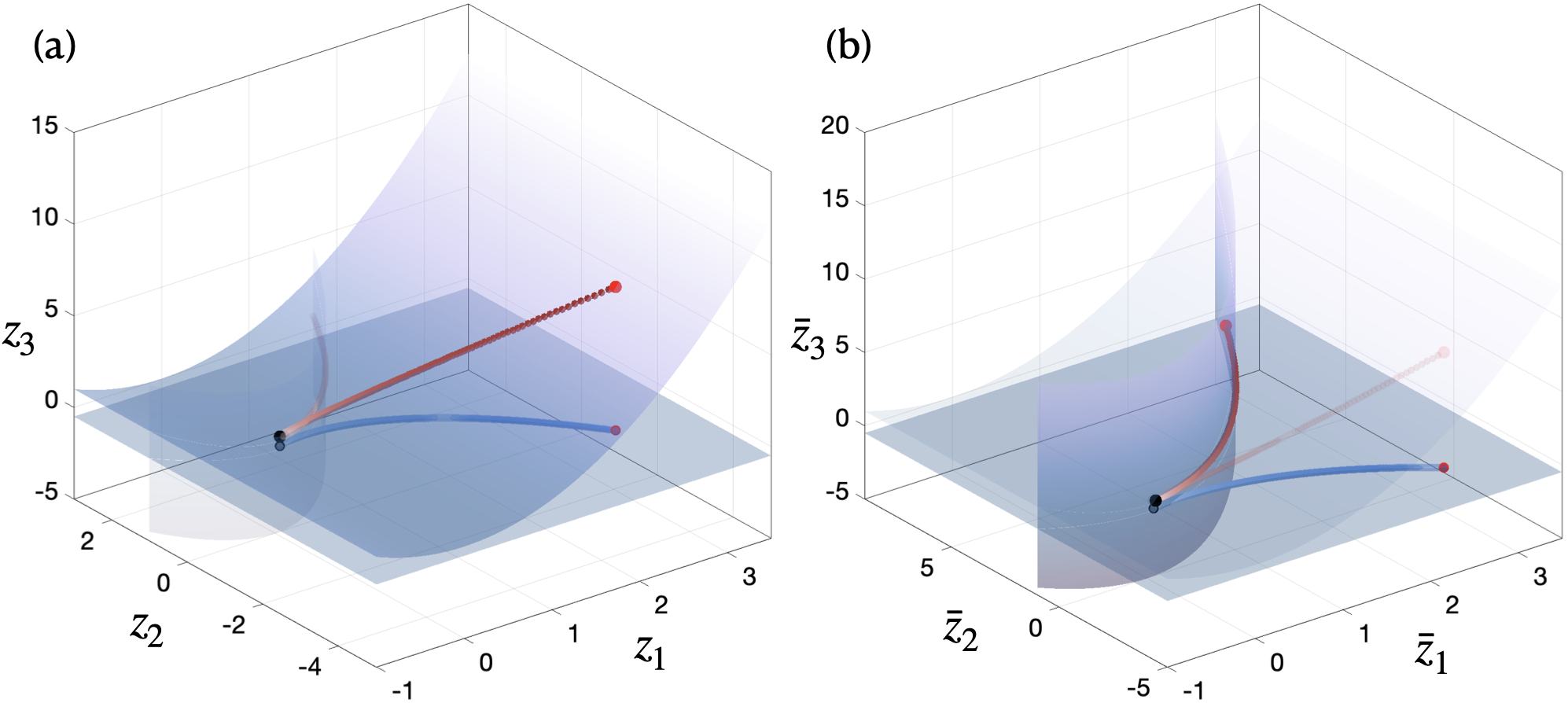}
    \caption{Illustration of the evolution (highlighted in blue) of the nonlinear control-affine system \eqref{eq:simulation1} and the evolution (highlighted in red) of its bilinear representations induced by different embeddings with (a) $\Psi(x) = (x_1,x_2,x_1^2)^{'}$ and (b) $\tilde{\Psi}(x) = (x_1,x_2-\lambda_3x_1^2,x_1^2)^{'}$.}
    \label{fig:simulationcompare}
\end{figure}

Finally, to demonstrate the feasibility of the EBIF framework for more complex systems, we present reachability analysis simulations using CORA \cite{Althoff_CORA15}, which computes reachable sets through zonotope-based evolution. The results are summarized in Figure \ref{fig:simulationtable} with their dynamics given in Table \ref{tab:systems}.

\begin{table*}[!h]
    \centering
    {\renewcommand{\arraystretch}{1.3}
    \begin{tabular}{|c|c|c|c|}
    \hline 
        & Systems dynamic & EBIF & equilibrium\\
    \hline
        (a) & { $\!\begin{aligned} 
        & f(x) = \begin{pmatrix}
        \lambda_1 x_1 \\ \lambda_2 x_2+(2\lambda_1-\lambda_2)\lambda_3x_1^2 
        \end{pmatrix}, \quad g(x) = \begin{pmatrix} 1 & 0\\ x_1^2&1  \end{pmatrix} \end{aligned}$ } & $\Gamma^* = \Gamma_0$ with $\Gamma_0 = \text{span}\{x_1,x_2,x_1^2\}$ & $ x_e = \begin{pmatrix} 0 \\ 0 \end{pmatrix}$ \\
    \hline
        (b) & { $\!\begin{aligned} 
        & f(x) = \begin{pmatrix} \lambda_1 x_2^3 \\ \lambda_2 x_3 \\ \lambda_3 \end{pmatrix} ,\quad  g(x) = \begin{pmatrix} 1 & 0 &0 \\ 0&1&0 \\ 0&0&1  \end{pmatrix} 
        \end{aligned}$ } & $\Gamma^* = \Gamma_4$ with $\Gamma_0 = \text{span}\{x_1,x_2,x_3\}$ & $ x_e = \begin{pmatrix} 0 \\ 0 \\0 \end{pmatrix}$ \\
    \hline
        (c) & { $\!\begin{aligned} 
        & f(x)= \begin{pmatrix} \lambda_1 x_1 \\ \lambda_2 x_2 + \lambda_3\cos(x_4) \\ \lambda_4x_3 +\lambda_5 \exp(x_4) \\ \lambda_6\end{pmatrix},~ g(x) = \begin{pmatrix} x_1 & 0&0&0 \\ 0&x_3&0&0 \\ 0&0&x_2&0 \\ 0&0&0&1  \end{pmatrix} 
        \end{aligned}$ } & $\Gamma^* = \Gamma_2$ with $\Gamma_0 = \text{span}\{x_1,x_2,x_3, x_4\}$ & $ x_e = \begin{pmatrix} 0.0442 \\ 0.0148 \\ -1.0992 \\ -0.4263 \end{pmatrix}$ \\
    \hline
    \end{tabular}
    }
    \caption{Details of nonlinear systems (in the form of $\dot{x}(t) = f(x) + g(x) u(t)$) used in Figure \ref{fig:simulationtable}.}\label{tab:systems}
\end{table*}

\begin{figure*}[!h]
\centering
\resizebox{\textwidth}{!}{
\begin{tabular}{|c|c|c|c|}
\hline
 & (a) & (b) & (c) \\
\hline
\rotatebox[origin=c]{90}{\textbf{reachable set}} & 
\adjincludegraphics[width=.3\linewidth, valign = c]{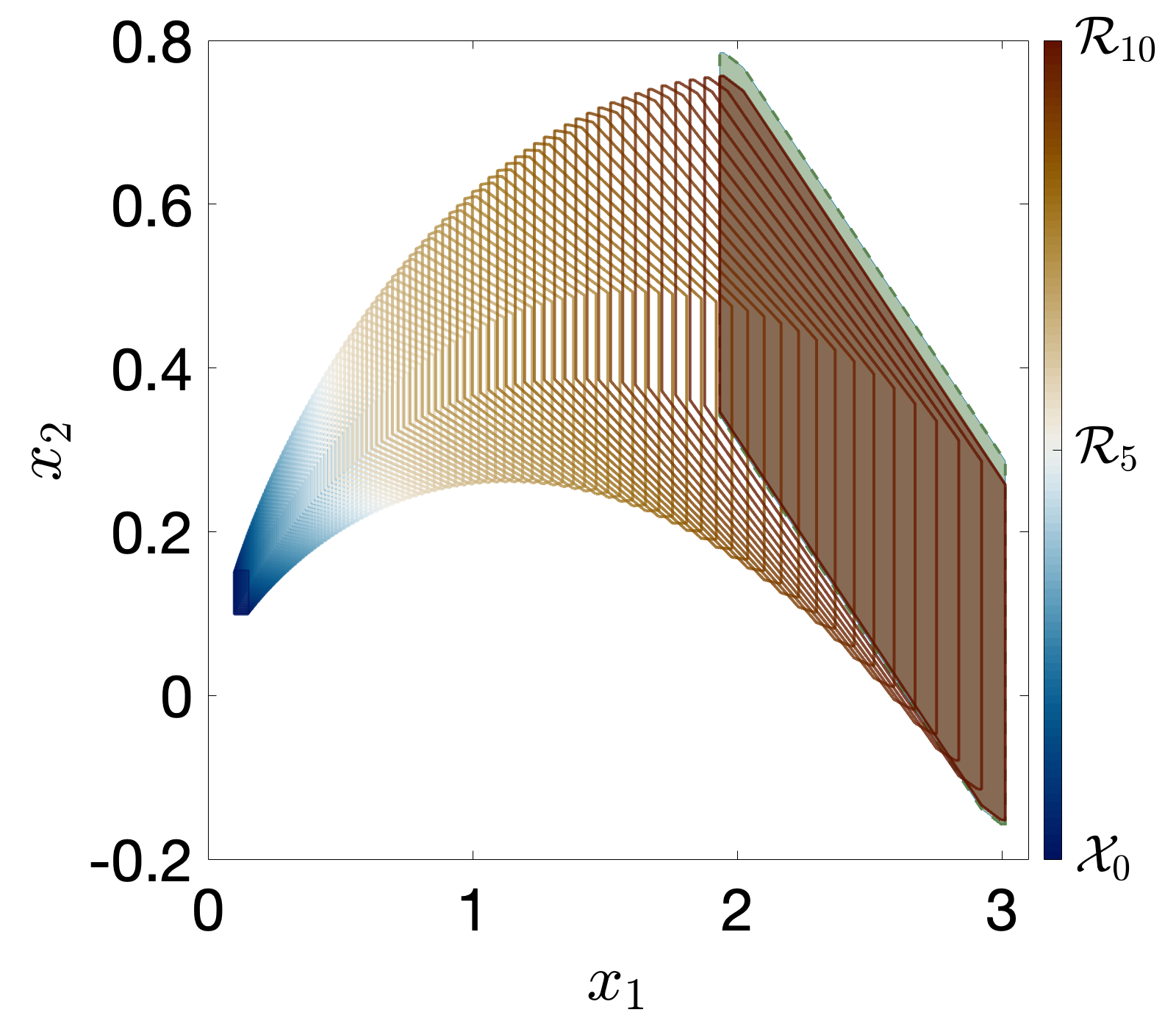} & 
\adjincludegraphics[width=.3\linewidth, valign = c]{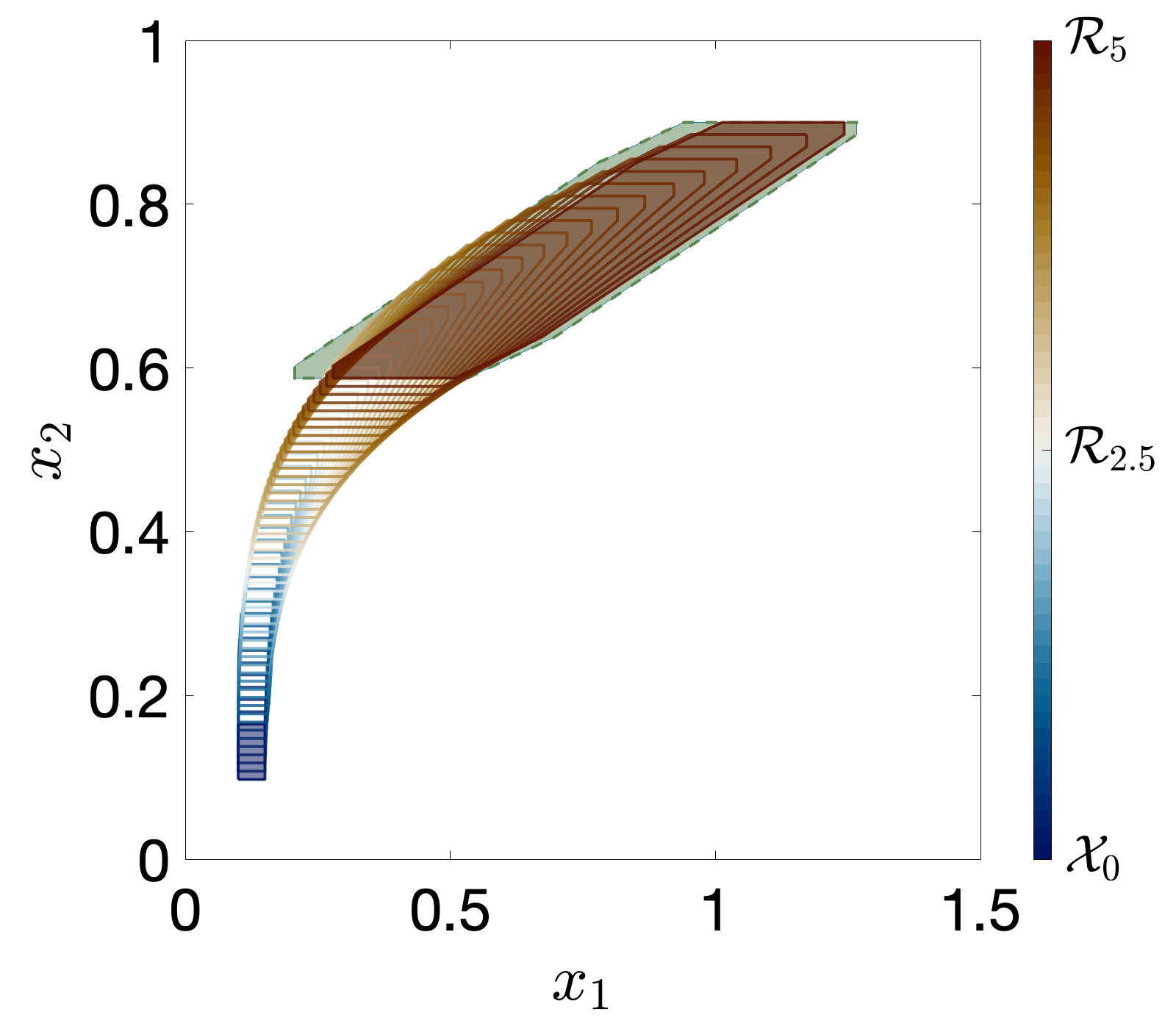} &
\adjincludegraphics[width=.3\linewidth, valign = c]{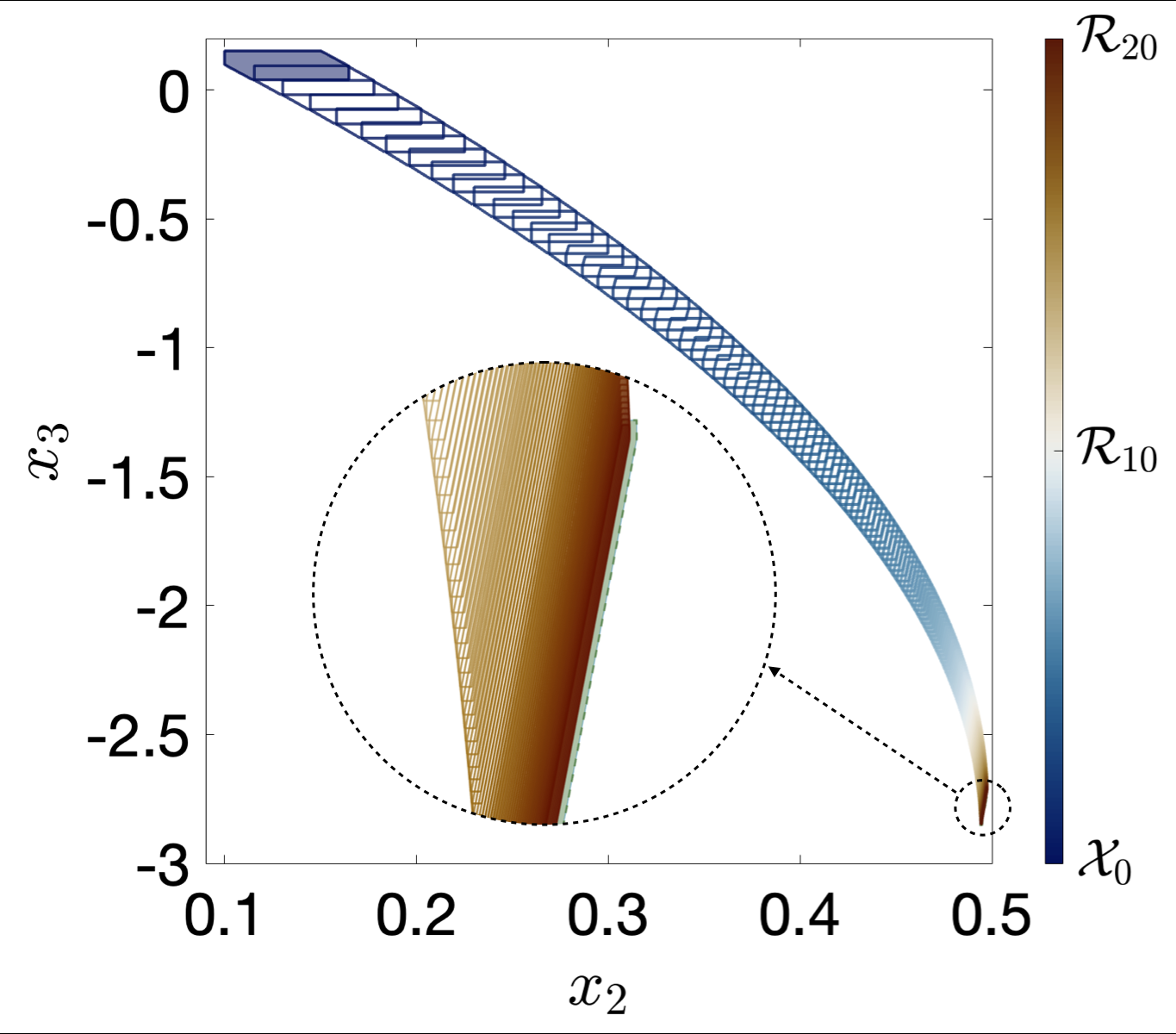}
\\
\hline
\rotatebox[origin=c]{90}{\textbf{stabilization}} & 
\adjincludegraphics[width=.3\linewidth, valign = c]{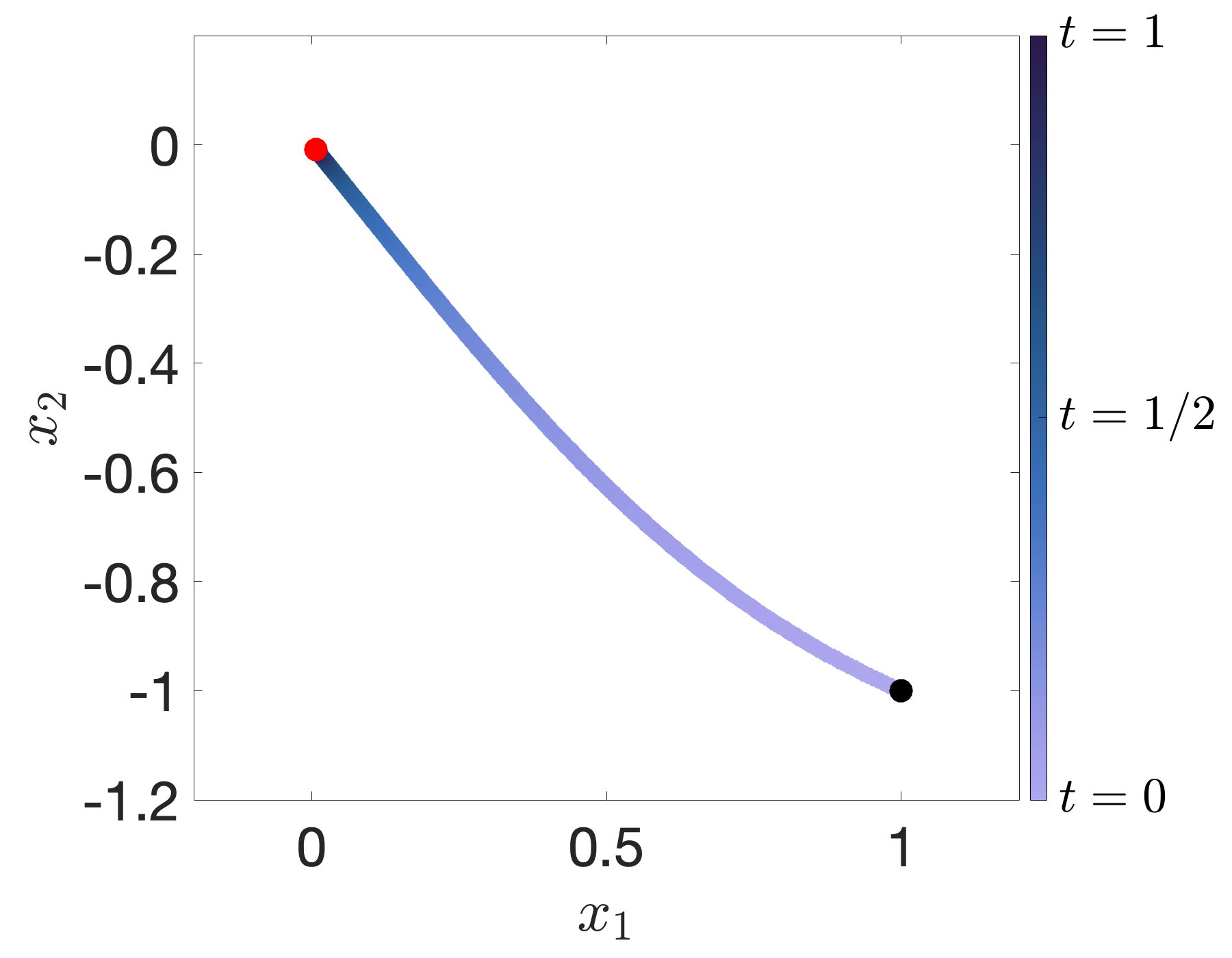} &
\adjincludegraphics[width=.3\linewidth, valign = c]{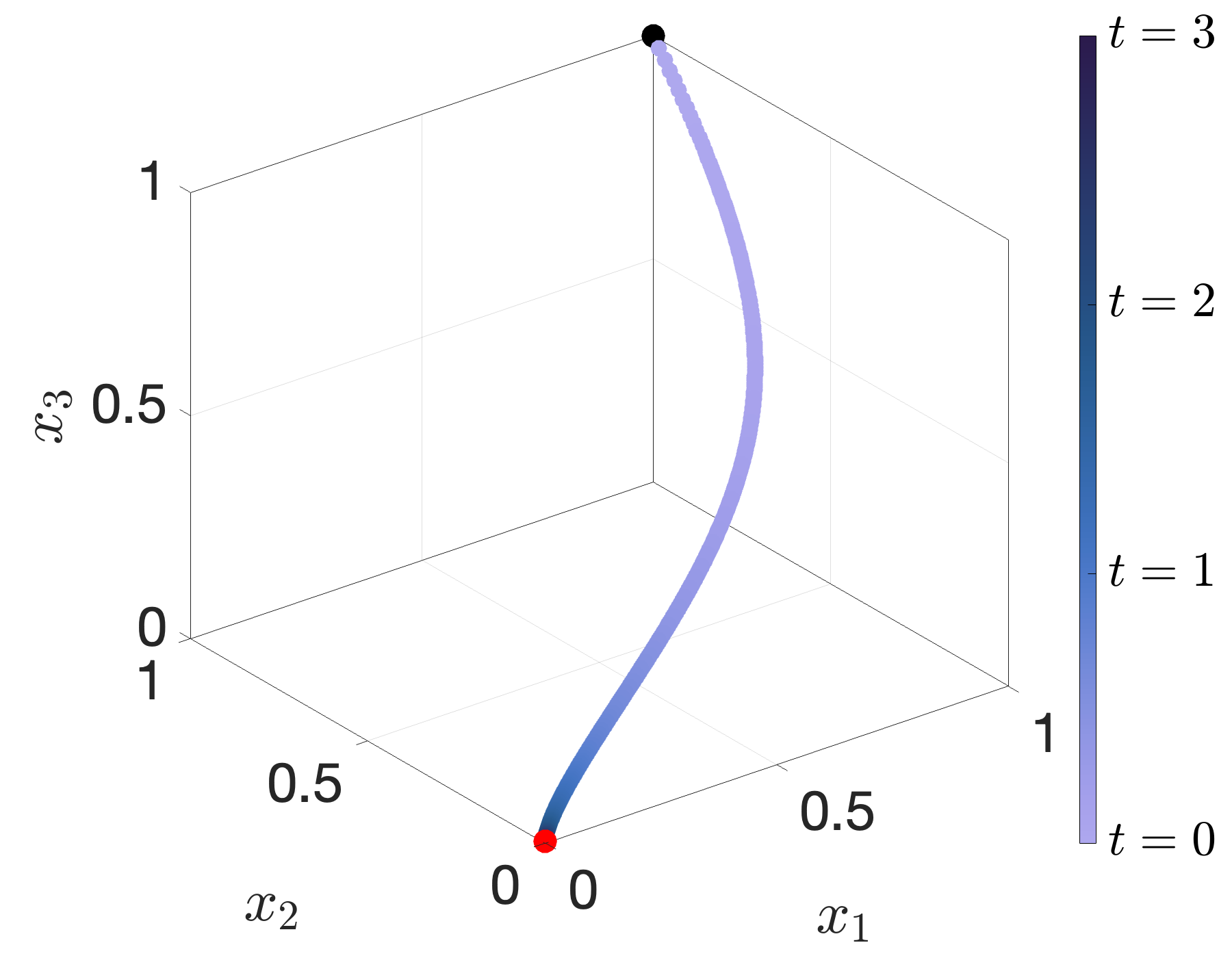} &
\adjincludegraphics[width=.3\linewidth, valign = c]{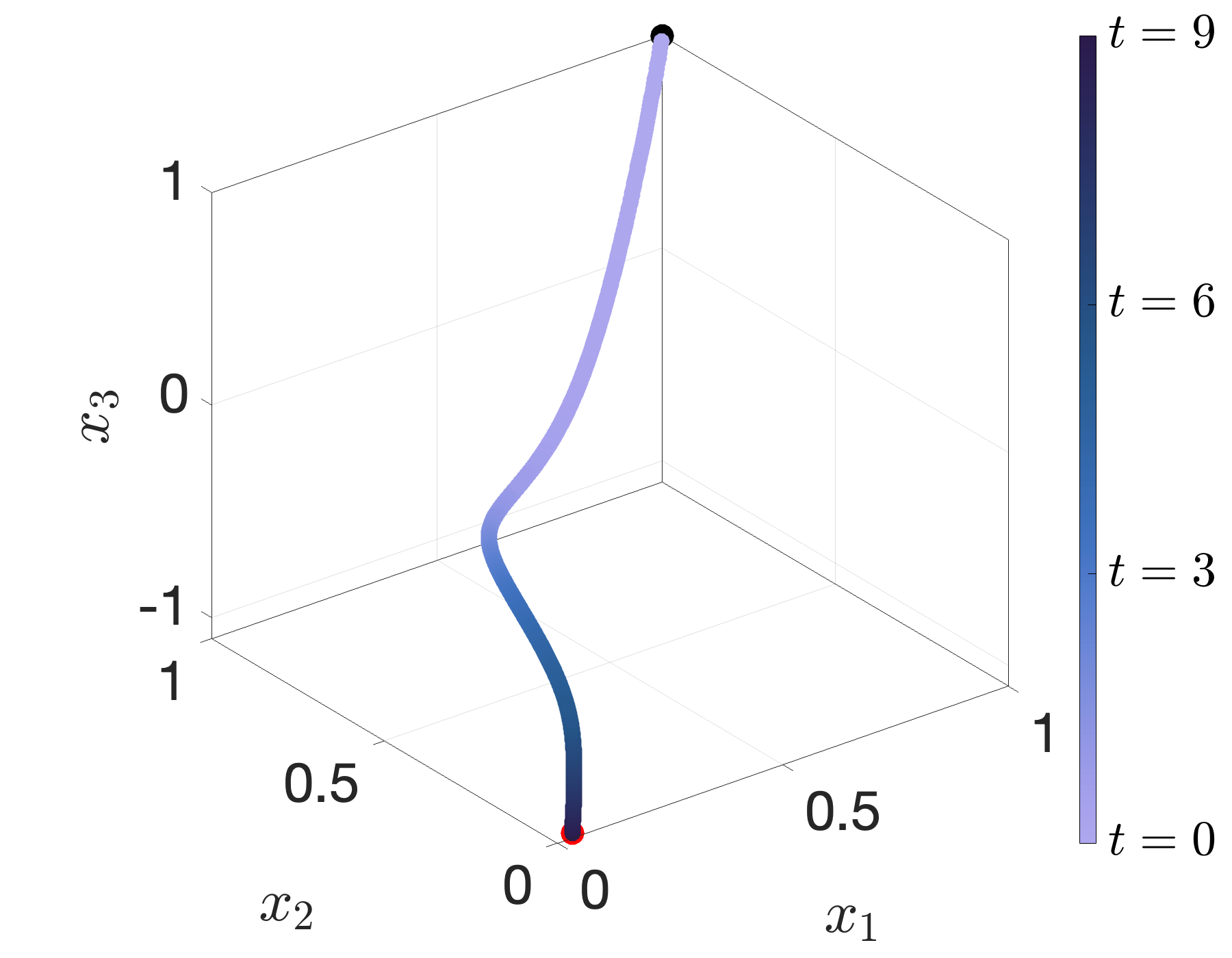}
\\
\hline
\end{tabular}
}
\caption{ Illustration of reachable set evolution and stabilize control design for the nonlinear control-affine systems in Table \ref{tab:systems}, with (a) $\lambda = \{0.3,0.2,-0.5\}$; (b) $\lambda = \{1,1,0\}$; and (c) $\lambda = \{-0.1,-0.4,0.2, -0.2,-0.5,0\}$ through their EBIF-induced bilinear systems. In the reachability simulations, we present the evolution of reachable set with no control input applied ($u(t)=0$). The initial sets (colored in blue) for all three systems are selected as $\mathcal{X}_0 = \{c+\sum_{i=1}^{n}\alpha_i h_i|\alpha_i\in[-1,1]\}$, with the center $c = (0.125,\ldots,0.125)'\in\mathbb{R}^n$ and the generators $h_i = 0.025e_i$ for all $i=1,\ldots,n$. The green region indicates the reachable set at the final time $T$ computed directly through the original nonlinear system from CORA, while the fading regions $\mathcal{R}_t$ illustrate the time-evolving reachable sets generated by the corresponding EBIF-induced bilinear system. In the stabilizing control design simulation, the goal is to design a feedback control through EBIF-enabled bilinear system to steer each system from a given initial (marked in black) to an equilibrium point $x_e$ (marked in red; values provided in Table \ref{tab:systems}).
}\label{fig:simulationtable}
\end{figure*}

\subsection{Stabilizing control design}

Designing stable control inputs for nonlinear systems is a critical challenge in control theory. The Lyapunov stability method provides a powerful framework for analyzing and ensuring system stability by constructing a suitable Lyapunov function that guarantees the system's trajectories remain bounded and converge to a desired equilibrium \cite{Gutman_TAC81,Amato_TCS09}. 
In the second scenario, we explore the Lyapunov-based techniques to design stable controllers for nonlinear systems through their bilinear systems.

To illustrate this concept, we consider the bilinear system 
\begin{align*}
    \dot{z}(t)&=Az(t)+\sum_{i=1}^m (B_iz(t)+D_i)u_i(t) := Az(t)+Su(t)
\end{align*}
associated with an exactly bilinearizable nonlinear control-affine system, where $S$ is an abbreviation denoting $S = S(z(t)) := [B_1z(t)+D_1,~\cdots, ~B_mz(t)+D_m ]$.
Assuming $P\succ0$ is a positive definite matrix with $S^{'}Pz \neq 0$, 
then there exists some positive definite matrix $K\succ 0$ such that the control in the form of
\begin{align*}
    u(t) = -KS^{'}(z(t))Pz(t)
\end{align*}
will stabilize the bilinear system. 
To see why this control leads to stabilize system, suppose the Lyapunov function $V:\mathbb{R}^r\to\mathbb{R}$ is a continuously differentiable defined as $V(z) = z^{'}Pz$. Substituting the given feedback control into the bilinear system dynamics yields $\dot{z}(t) = Az(t) - S(z(t))KS^{'}(z(t))Pz(t)$,
then taking the directional derivative of $V$ along the trajectory of $z$ yields
\begin{align*}
    \dot{V} 
    &= z'\left(A'P+PA\right)z - 2\|z' PS\sqrt{K}\|^2, 
\end{align*}
where $\|\cdot\|$ denotes the Euclidean norm of the given vector and $\sqrt{K}$ is the matrix satisfying $\sqrt{K}'\sqrt{K} = K$ by positive definiteness of matrix $K$.
It follows immediately that if $z$ satisfies $z^{'}(A^{'}P+PA)z < 0$ then clearly we have $\dot{V}<0$ as desired. On the other hand, for $z^{'}(A^{'}P+PA)z\geq 0$ with $z\neq 0$, denote $Q := A^{'}P+PA$, then $\dot{V}<0$ implies $z^{'}Qz - 2\|z^{'} PS\sqrt{K}\|^2 < 0$ and therefore
\begin{align*}
    \frac{\|z^{'}Qz\|}{\|z^{'}z\|} 
    &< 2\frac{\|z^{'} PS\sqrt{K}\|^2}{\|z^{'}z\|} \leq 2C_K^2 c_2\|z^{'}z\| 
\end{align*}
where $C_K = \|\sqrt{K}\|$ denotes the matrix norm of $\sqrt{K}$ and $c_2$ is constant depending on $P$, $B_i$, and $D_i$ (see Lemma \ref{lemma:stable} in Appendix for details). This indicates by choosing the matrix $K\succ 0$ with
\begin{align*}
    C_K = \|\sqrt{K}\| > \sqrt{\frac{\kappa(Q,P,B_i,D_i)}{2m\varepsilon^2}},
\end{align*}
where $\kappa(Q,P,B_i,D_i)$ is a constant depending on $Q,P,B_i,D_i$ as provided in Lemma \ref{lemma:stable}, then $V$ is guaranteed to satisfy $\dot{V}< 0$ and thus $z$ is stabilized to an equilibrium point as desired.

Despite the theoretical bound provided in the above discussion, in practice, the construction can be simplified to avoid the potentially cumbersome design of $K$. 
One practical approach is to choose $K$ as a diagonal matrix with all positive diagonal entries, which facilitates one to construct $K$ satisfying the required conditions. Alternatively, one may directly select $P$ such that $Q = A^{'}P+PA \prec 0$. 
Utilizing the above result, we apply the stabilize control
\begin{align}
    u_{stable}(t) = -K\begin{pmatrix}
        (B_1\Psi(x(t))+D_1)^{'}\\
        \vdots\\
        (B_m\Psi(x(t))+D_m)^{'}
    \end{pmatrix} P\Psi(x(t))
\end{align}
to the nonlinear system $\dot{x}(t) = f(x) + \sum_{i=1}^{m}g_i(x) u_i(t)$ and present the stabilizing results of the three systems in Figure \ref{fig:simulationtable}.



\subsection{Numerical control design for optimal control problems}

Leveraging the finite-dimensional bilinear representation provided by EBIF, in the last section, we explore the optimal control framework induced by the associated bilinear embedding. 
To illustrate our main idea, we consider the optimal control problem in Bolza type \cite{Berkovitz_CRC13} with both running costs and terminal costs in quadratic form, given by
\begin{align} \label{eq:optimalcontrol}
    \min_{u\in\mathcal{U}}&\quad J(u)= \int_0^T L(x(t),u(t))dt + \varphi(x(T)), \\
    {\rm s.t.}&\quad \dot{x}(t) = f(x(t)) +\sum_{i=1}^{m} u_i(t)g_i(x(t)) ,~x(0)=x_0, \nonumber
\end{align}
where $L(x(t),u(t)) =x^{'}Kx + u^{'}Ru$ and $\varphi(x(T)) = x^{'}Qx$ for some $K,Q\succeq 0$ and $R\succ0$, and the space of control input $\mathcal{U} = \{u:[0,T]\to \Omega\subset \mathbb{R}^m\}$ with some compact set $\Omega$. Note that the existence of an optimal solution to \eqref{eq:optimalcontrol} is guaranteed following that $L$ is continuously differentiable on $\mathbb{R}^n\times\mathbb{R}^m$ , $\varphi$ is lower semi-continuous function, and $u$ takes values from a compact domain \cite{Berkovitz_CRC13}. 
Applying the smooth embedding constructed by EBIF, the associated cost functions remains in quadratic form as $J(u)= \int_0^T \bar L(z(t),u(t))dt + \bar{\varphi}(z(T))$, with $\bar L(z(t),u(t)) =z^{'}\bar Kz + u^{'}\bar Ru$, $\bar\varphi(z(T)) = z^{'}\bar Qz$, and $(\bar K, \bar R,\bar Q) = ({P}_n^{'}K{P}_n, R, {P}_n^{'}Q{P}_n)$.
This quadratic preserving property enables us to further derive the explicit form of optimal control by solving the optimality condition $\nabla_{u_k}\bar{\mathcal{H}}(z(t),u(t),\bar{\lambda}(t)) = 0$ for each $k=1,\ldots,m$, where $\bar{\mathcal{H}}$ is the Hamiltonian to the bilinear system $\Sigma_b$ and $\bar{\lambda}(t)$ denotes the costate trajectory satisfying $\dot{\bar \lambda}(t) = - \nabla_{z} \bar{\mathcal{H}}(z(t),u(t),\bar{\lambda}(t))$ with $\bar{\lambda}(T) = \nabla_{z} \varphi(z(T))$. This leads to the explicit form of the optimal control as
    \begin{align}
        u_{optimal}(t) = \frac{1}{2}R^{-1}
        \begin{pmatrix}
            \bar{\lambda}(t)^{'}B_1\Psi(x(t))\\
            \vdots \\
            \bar{\lambda}(t)^{'}B_m\Psi(x(t))
        \end{pmatrix}.
    \end{align}

In addition to this explicit solution to the optimal control problem with quadratic costs, to demonstrate the power of EBIF in a numerical setting, we now drop the running cost in the objective functional and consider only the terminal cost. That is, we consider the optimal control problems of Mayer type with the objective function given by $J(u)= \varphi(x(T))$. This omission is valid in the sense that these two types are interchangeable with the same minimum cost \cite{Bliss_46, Fleming_Springer12}, providing the running cost $L(x,u)$ is Lipschitz continuous in the first argument uniformly for any $u(t)\in\mathbb{R}^m$.
Given the fact that the reachable sets of a control system driven by bounded continuous inputs and piecewise constant control inputs are identical \cite{Jurdjevic_Cambridge97}, we focus on designing piecewise constant control inputs for the above optimal control problem. That is, we consider a partition $0=t_0\leq t_1\leq\cdots \leq t_s = T$ of the time horizon $[0,T]$ and take the control in each time interval $[t_k,t_{k+1})$ as $u(t) = u(t_k) := \mathfrak{u}_k$, 
where $\mathfrak{u}_k$ denotes the constant control applied during the $k$th subinterval with $k=0,1,\cdots,s-1$. In this case, the trajectory of the system satisfies
\begin{align} \label{eq:piecewisecontrol}
    x(T)  =\mathcal{P}_n\big( \Phi(\mathfrak{u}_{s-1})\cdots\Phi(\mathfrak{u}_{1})\Phi(\mathfrak{u}_{0})\Psi(x_0) \big),
\end{align}
where $\Phi(\mathfrak{u}_{k}):=\exp\{\delta_k(A + \sum\nolimits_{i=1}^{m} \mathfrak{u}_{ik} B_i )\}$ denotes the flow driven by piecewise constant control $\mathfrak{u}_{k}:=(\mathfrak{u}_{1k},\ldots,\mathfrak{u}_{mk})^{'}$ in the interval $[t_k,t_{k+1})$ and $\delta_k = t_{k+1} - t_{k}$.
One direct advantage of the above explicit expression \eqref{eq:piecewisecontrol} is that the optimal control problem \eqref{eq:optimalcontrol} can be transformed into an unconstrained optimization problem with the objective function being $ J(u)= \varphi\Big( \mathcal{P}_n\big( \Phi(\mathfrak{u}_{s-1})\cdots\Phi(\mathfrak{u}_{1})\Phi(\mathfrak{u}_{0})\Psi(x_0) \big) \Big)$ and thus
\begin{align} \label{eq:control_numerical}
    u_{num} = \underset{\mathfrak{u}_{k} \in\mathbb{R}^m}{\arg\min}~ \varphi\Big( \mathcal{P}_n\big( \Phi(\mathfrak{u}_{s-1})\cdots\Phi(\mathfrak{u}_{0})\Psi(x_0) \big) \Big). 
\end{align}
This formulation enables one to adopt any standard optimization solvers to efficiently solve the optimal control problem in \eqref{eq:optimalcontrol}. To demonstrate the practical utility of this numerical approach, we present a steering simulation involving 250 nonlinear unicycle systems governed by the dynamics described in Example \ref{ex:unicycle_derive}. The objective function is set as $J = \|x(T) - x_F\|^2$, where $x_0$ and $x_F$ are the given initial and final configurations, forming the shapes of the patterns ``WUSTL'' and ``AMLAB,'' respectively. The goal is to steer each system from the assigned initial to the desired final using the designed control strategy, and the result is demonstrated in Figure \ref{fig:unipattern}.
\begin{figure}[H]
    \centering
    \includegraphics[width=\linewidth]{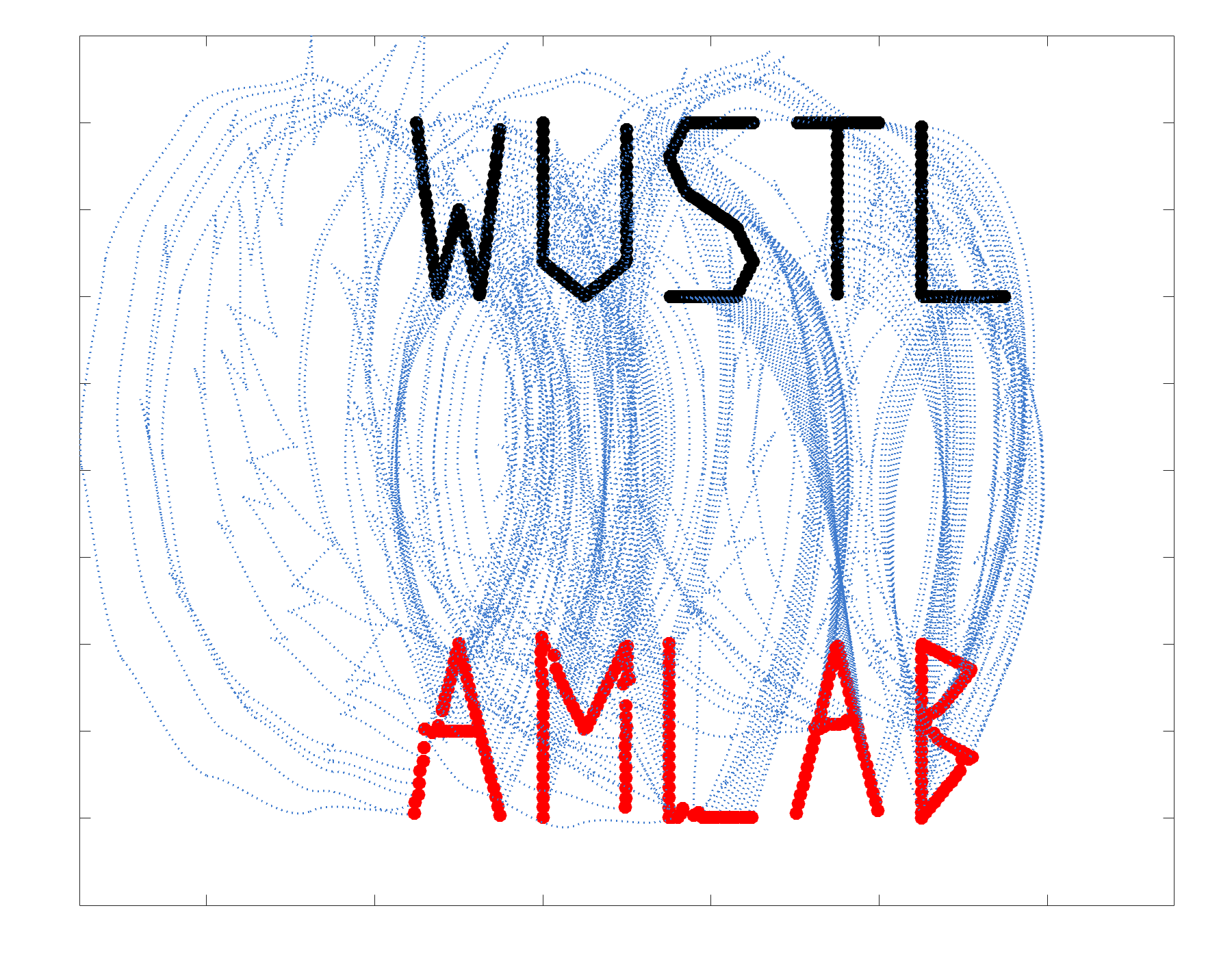}
    \caption{Illustration of optimal steering problem of the nonlinear control-affine system in Example \ref{ex:unicycle_derive}, with the objective function $J = \|x(T) - x_F\|^2$. Initial and final states of each system are represented by black and red dots, respectively. The dashed blue lines indicate the trajectories of individual systems under the numerical control obtained in \eqref{eq:control_numerical}.}
    \label{fig:unipattern}
\end{figure}

\section{Conclusion} \label{sec:conclusion}

In this paper, we introduce EBIF, a novel and systematic framework for transforming a nonlinear  control-affine system into an exact, finite-dimensional bilinear representation.
Utilizing the definition of direct limit and leveraging the concept of vector space Noetherian, we derive both necessary and sufficient conditions for a nonlinear system to be exactly bilinearizable.
This new paradigm not only formalizes the exact bilinearization process through an iterative algorithm, but also ensures the equivalence between the original nonlinear system and its bilinear counterpart, providing the constructed bilinear map is a smooth embedding. 
These EBIF-induced transformations enable the use of bilinear system theory to facilitate reachability analysis, optimal control design, and stabilization in a more computationally tractable manner.
Simulations across various nonlinear dynamics validate the practical effectiveness of EBIF, highlighting its applicability in enabling both theoretical and numerical control synthesis for complex dynamical systems.


\appendices
\section{ }     

\begin{lemma} \label{lemma:coord}
    Suppose $\Psi:\mathbb{R}^n \to \mathbb{R}^r$ is a smooth map from $\mathbb{R}^n$ to $\mathbb{R}^r$, and let $x=(x_1,\dots,x_n)$ and $z=(z_1,\dots,z_r)$ be the coordinates on $\mathbb{R}^n$ and $\mathbb{R}^r$, respectively. Then,
    \begin{align*}
        \bar{f}=\Psi_{\ast}f = \sum_{j=1}^r\sum_{l=1}^n f_l   \frac{\partial \Psi_j}{\partial x_l} \frac{\partial }{\partial z_j}
    \end{align*}
    where $f$ and $\bar f$ are the vector fields on $\mathbb{R}^n$ to $\mathbb{R}^r$, respectively, and $f=\sum_{l=1}^nf_l\frac{\partial}{\partial x_l}$ are the coordinate representations of $f$.
\end{lemma}
\begin{proof}
    For any smooth function $h\in \mathcal{C}^{\infty}(\mathbb{R}^r)$ on $\mathbb{R}^r$, it follows that
    \begin{align*}
        (\Psi_{\ast}&f)h = 
        \Psi_{\ast}\left(\sum_{l=1}^n f_l\frac{\partial}{\partial x_l}\right) h = \sum_{l=1}^n f_l\left( \Psi_{\ast}\frac{\partial}{\partial x_l} \right)h \\
        &= \sum_{l=1}^n f_l\left( \frac{\partial}{\partial x_l} \right)( h\circ \Psi) = \sum_{l=1}^n f_l  \frac{\partial ( h\circ \Psi)}{\partial x_l}  \\
        &= \sum_{l=1}^n f_l \left( \sum_{j=1}^r \frac{\partial h}{\partial z_j} \frac{\partial \Psi_j}{\partial x_l} \right) = \sum_{j=1}^r\sum_{l=1}^n f_l   \frac{\partial \Psi_j}{\partial x_l} \frac{\partial }{\partial z_j} h,
    \end{align*}
    which leads to the desired result.
\end{proof}

\begin{lemma} \label{lemma:Kstationaryproof}
    Let $\tau$ be a smooth vector field and $\Gamma = \langle \gamma_1, \ldots, \gamma_d\rangle_{\mathbb{A}}$ be a free module generated by a set of smooth functions $\{\gamma_i| \gamma_i\in C^{\infty}, i = 1,\ldots, d\}$, then for any smooth function $h \in \Gamma$, it holds that
     \begin{align*}
         \mathcal{L}_{\tau}h \in \mathcal{L}_{\tau}\Gamma.
     \end{align*}
     If further $\Gamma$ is invariant under $\tau$, then $\mathcal{L}_{\tau}h \in\Gamma$.
\end{lemma}
\begin{proof}
    Suppose $h \in \Gamma$, then there exists $k_1,\ldots,k_d \in \mathbb{A}$ such that $h = \sum_{i=1}^{d} k_i\gamma_i$, then by linearity of Lie derivative, 
    \begin{align*}
        \mathcal{L}_{\tau}h(x) = \frac{\partial h}{\partial x}\tau(x) = \sum_{i=1}^d k_i\frac{\partial \gamma_i}{\partial x}\tau(x) = \sum_{i=1}^d k_i \mathcal{L}_{\tau} \gamma_i(x),
    \end{align*}
    which is an element in $\mathcal{L}_{\tau}\Gamma$ (by Definition \ref{def:LtauGamma}). If further $\Gamma$ is invariant under $\tau$, then $\mathcal{L}_{\tau}\Gamma\subset \Gamma$ and thus $\mathcal{L}_{\tau}h \in\Gamma$.
\end{proof}


\begin{lemma} \label{lemma:brockettreachability}
    Given a bilinear system $\Sigma_b$ in \eqref{eq:Sigma_b} and consider the matrix differential equation in $\mathbb{R}^{r\times r}$ associated with $\Sigma_b$, given by
    \begin{align*}
        \bar{\Sigma}_b:\dot {\Phi}(t) = A\Phi(t) + \sum\nolimits_{i=1}^{m} B_i\Phi(t) u_i(t). 
    \end{align*}
    Assume $[\text{ad}_A^k B_i,B_j]=0$ for all $i,j=1,\ldots,m$ and let $\mathcal{H} = \text{span}\{ \text{ad}_A^k B_i\}$ for $k= 0,\ldots,r^2-1$, then there exists continuous controls which steer from $\Phi(0) = X_0$ to $\Phi(T) = X_1$ at some final time $T$ if and only if there exists $H\in\mathcal{H}$ such that
    $X_1 = e^{AT}e^H X_0$.
\end{lemma}
\begin{proof}
    See \cite{Brockett_SIAM72} for details. 
\end{proof}

\begin{lemma}\label{lemma:dimpreserve}
    Let $\Psi:\mathbb{R}^n \to \mathbb{R}^r$ be a smooth map from $\mathbb{R}^n$ to $\mathbb{R}^r$, and suppose $f_i$ and $\bar f_i$ are $\Psi$-related vector fields for $1\leq i \leq m$, then the following holds
    $$[\bar{f}_i , \bar{f}_j]_{\Psi(x)} = \Psi_{\ast} [f_i , f_j]_x,$$
    for all $0\leq i\neq j \leq m$.
\end{lemma} 
\begin{proof}
    Since the smooth vector fields $f_i\in\mathfrak{X}(\mathbb{R}^n)$ and $\bar{f}_i\in\mathfrak{X}(\mathbb{R}^r)$ are $\Psi$-related, for any smooth function $h\in\mathcal{C}(\mathbb{R}^n)$, we have
    \begin{align} \label{proof:dimension}
        [\bar{f}_i , \bar{f}_j]_{\Psi(x)}h &= \bar{f}_i|_{\Psi(x)}(\bar{f}_jh) - \bar{f}_j|_{\Psi(x)}(\bar{f}_ih) \nonumber\\
        &=(\Psi_{\ast}f_i|_x)\bar{f}_jh - (\Psi_{\ast}f_j|_x)\bar{f}_ih \nonumber\\
        &=f_i|_x ((\bar{f}_j h)\circ\Psi) - f_j|_x ((f_i h)\circ\Psi) \nonumber\\
        &=f_i|_x (f_j (h\circ\Psi)) - f_j|_x (f_i (h\circ\Psi)) \\
        &=[f_i, f_j]_x (h\circ\Psi) =\Psi_{\ast} [f_i, f_j]_x h \nonumber
    \end{align}
    where the equality in \eqref{proof:dimension} follows from the fact that $(\bar{f}_j h)\circ\Psi|_x = (\bar{f}_j h)|_{\Psi(x)} = \bar{f}_j|_{\Psi(x)}h = (\Psi_{\ast} f_j |_x )h = f_j|_x(h\circ\Psi) = f_j(h\circ\Psi) |_x$.
\end{proof}


\begin{lemma} \label{lemma:stable}
    Consider the bilinear system 
    $
        \dot{z}(t) = Az(t)+\sum_{i=1}^{m} u_i(t)(B_iz+D_i).
    $
    Define the Lyapunov function $V:{ N }\to\mathbb{R}$ as $V(z) = z^{'}Pz$ for some positive definite matrix $P\succ 0$, then the feedback control in the form of 
    \begin{align*}
        u(t) = -KS^{'}(z(t))Pz(t)
    \end{align*}
    will stabilize the bilinear system if $K$ is a positive definite matrix satisfying
    \begin{align*}
        C_k = \|\sqrt{K}\| > \sqrt{\frac{\lambda_{\min}(Q)}{2 m \lambda_{\max}(\bar{P}_\ell)^2 \varepsilon^2}},
    \end{align*}
    where $\sqrt{K}$ is the matrix satisfying $\sqrt{K}^{'}\sqrt{K} = K$, $Q = A^{'}P+PA$, and $\bar{P}_\ell$ is a matrix satisfying $\lambda_{\max}(\bar{P}_\ell) = \max \{ \lambda_{\max}(\begin{pmatrix} PB_i & \frac{1}{2}PD_i \\ \frac{1}{2}D_i^{'}P^{'} & 1 \end{pmatrix}) |~ i=1,\ldots,m\}$.
\end{lemma}
\begin{proof}
    Let $S := [B_1z+D_1,~\cdots, ~B_mz+D_m ]$. The condition $\dot{V}<0$ implies 
    \begin{align*}
        &\|z^{'}Qz\| 
        < 2 \|z^{'} PS\sqrt{K}\|^2 \leq 2C_K^2 \|z^{'} PS\|^2 \\
        & = 2C_K^2 \sum_{i=1}^{m}\|z^{'} P(B_iz+D_i)\|^2 = 2C_K^2 \sum_{i=1}^{m} \left\|\bar{z}^{'} \bar{P}_i \bar{z} -1 \right\|^2, 
    \end{align*}
    where $\bar{z} = (z^{'}, 1)^{'}$ and
    $\bar{P}_i=\begin{pmatrix}
            PB_i & \frac{1}{2}PD_i \\ \frac{1}{2}D_i^{'}P^{'} & 1
        \end{pmatrix}.$
    Then
    \begin{align*}
        & \lambda_{\min}(Q) \leq \frac{\|z^{'}Qz\|}{\|z^{'}z\|} < 2C_K^2 \sum_{i=1}^{m} \frac{ \left\|\bar{z}^{'} \bar{P}_i \bar{z} -1 \right\|^2 }{\|z^{'}z\|} \\
        &= 2C_K^2 \sum_{i=1}^{m} \frac{ \left\|\bar{z}^{'} \bar{P}_i \bar{z} -1 \right\|^2 }{\|\bar{z}^{'}\bar{z} -1\|}  \leq 2C_K^2 \sum_{i=1}^{m} \frac{ \left\|\bar{z}^{'} \bar{P}_i \bar{z} \right\|^2 }{\|\bar{z}^{'}\bar{z} \|}\\
        & \leq 2C_K^2 \left(\sum_{i=1}^{m}\lambda_{\max}(\bar{P}_i)^2 \right)\| \bar{z}^{'}\bar{z} \|
    \end{align*}
    where the first and last inequality bound conditions are given by the Rayleigh quotient of a real symmetric matrix and $\lambda_{\min}(\cdot)$ and $\lambda_{\max}(\cdot)$ return the smallest and the largest eigenvalues of the underlying symmetric matrix, respectively. If we further denote $\lambda_{\max}(\bar{P}_\ell) = \max \{ \lambda_{\max}(\bar{P}_i) |~ i=1,\ldots,m\}$, then 
    \begin{align*}
        & \lambda_{\min}(Q)  < 2C_K^2 m \lambda_{\max}(\bar{P}_\ell)^2 \| {z}^{'} {z} \|
    \end{align*}
    which leads to the desired bound condition.
\end{proof}





\bibliographystyle{ieeetr}
\bibliography{references}


\end{document}